\documentclass[12pt]{amsart}

\usepackage{fullpage, graphicx, enumitem}
\usepackage{url}
\usepackage{verbatim}
\usepackage{bbm}
\usepackage{amssymb,amsmath}
\usepackage{amsfonts,savesym,graphicx,bm,amsthm,stmaryrd}
\usepackage{color}
\usepackage[mathscr]{eucal}
\usepackage[all]{xy}
\usepackage{mathtools}

\definecolor{hot}{RGB}{65,105,225}
\usepackage[pagebackref=true,colorlinks=true, linkcolor=hot ,  citecolor=hot, urlcolor=hot]{hyperref}

\newtheorem{theorem}{Theorem}[section]
\newtheorem{lemma}[theorem]{Lemma}

\newtheorem{theorem-definition}[theorem]{Theorem-Definition}
\newtheorem{corollary}[theorem]{Corollary}
\newtheorem{proposition}[theorem]{Proposition}
\newtheorem{prop}[theorem]{Proposition}

\newtheorem{ques}[theorem]{Question}
\newtheorem{definition-theorem}[theorem]{Definition-Theorem}

\newtheorem{theorem-defintion}[theorem]{Theorem-Definition}
\newtheorem{corollary-definition}[theorem]{Corollary-Definition}
\newtheorem{definition-proposition}[theorem]{Definition-Proposition}

\theoremstyle{definition}

\newtheorem{definition}[theorem]{Definition}

\newtheorem{remark}[theorem]{Remark}

\newtheorem{subs}[theorem]{ }
\newtheorem{rmk}[theorem]{Remark}

\numberwithin{equation}{section}

\numberwithin{equation}{section}

\def\bC{\mathbb{C}}
\def\be{\begin{equation}}
	\def\ee{\end{equation}}

\def\bZ{\mathbb{Z}}
\def\al{\alpha}
\def\bQ{\mathbb{Q}}

\def\ord{\mathrm{ord}}
\def\Spec{\mathrm{Spec}\,}

\def\cL{\mathcal L}

\def\Supp{\mathrm{Supp}}

\def\Ex{{\mathrm{Ex}}}
\def\bL{\mathbb{L}}
\def\bA{\mathbb{A}}
\def\cD{\mathcal{D}}
\def\sX{\mathscr{X}}
\def\sL{\mathscr{L}}
\def\ol{\overline}
\def\sY{\mathscr{Y}}
\def\bP{\mathbb{P}}
\def\Homo{\mathrm{Hom}}
\def\Var{\mathrm{Var}}
\def\Bir{\mathrm{Bir}}
\def\Sk{\mathrm{Sk}}
\def\cF{\mathcal{F}}

\def\lct{\mathrm{lct}}

\def\str{\mathrm{str}}
\def\exc{\mathrm{exc}}
\def\dlt{\mathrm{dlt}}
\def\red{\mathrm{red}}
\def\mot{\mathrm{mot}}
\def\top{\mathrm{top}}
\def\bir{\mathrm{bir}}
\def\rat{\mathrm{rat}}

\author{Tom Biesbrouck}
\address{Department of Mathematics, KU Leuven, Celestijnenlaan 200B, 3001 Leuven, Belgium.}
\email{tom.biesbrouck@kuleuven.be}

\author{Nero Budur}
\address{Department of Mathematics, KU Leuven, Celestijnenlaan 200B, 3001 Leuven, Belgium;  YMSC, Tsinghua University, 100084 Beijing, China;  BCAM, Mazarredo 14, 48009 Bilbao, Spain.}
\email{nero.budur@kuleuven.be}

\author{Johannes Nicaise}
\address{Department of Mathematics, KU Leuven, Celestijnenlaan 200B, 3001 Leuven, Belgium.}
\email{johannes.nicaise@kuleuven.be}

\author{Willem Veys}
\address{Department of Mathematics, KU Leuven, Celestijnenlaan 200B, 3001 Leuven, Belgium.}
\email{wim.veys@kuleuven.be}

\title{Birational zeta functions}

\begin{document}
\begin{abstract}
 We define a birational analog of the motivic zeta function of a reduced polynomial in terms of minimal models. It admits an intrinsic meaning in terms of contact loci of arcs, an analog of a result of Denef and Loeser in the motivic case. We show that  for local plane curve singularities the poles of the birational zeta function essentially coincide with the poles of the motivic zeta function.
  \end{abstract}

\maketitle

\section{Introduction}

\subs{\bf Motivation.}\label{subsMot} Let $f\in\bC[x_1,\ldots,x_n]$ be a non-constant polynomial. Associated with $f$ one has the {\it contact loci} $\sX_m(f)$, consisting of $m$-jets with order precisely $m$ along $f$, for $m\in\bZ_{>0}$.
The  {\it motivic zeta function} of $f$ is 
$$
Z^{\mot}_{f}(T)\vcentcolon= \sum_{m\ge 1}[\sX_m(f)]\bL^{-mn}T^m\quad \in K_0(\Var_\bC)[\bL^{-1}]\llbracket T \rrbracket,
$$
a normalized generating series for the classes in the Grothendieck ring of complex varieties of the contact loci, where $\bL=[\bA^1]$. Denef and Loeser \cite{DL} showed that it can be expressed as a rational function in terms of any log resolution $\mu:Y\to \bA^n$ of $f$:
$$
Z^{\mot}_{f}(T)=  \sum_{\emptyset\neq I\subset S}[E_I^\circ]\prod_{i\in I}\frac{\bL-1}{\bL^{\nu_i}T^{-N_i}-1},
$$
where $(f\circ\mu)^{-1}(0)=\sum_{i\in S}N_i E_i$ is a simple normal crossing divisor with irreducible components indexed by $S$, $N_i$ is the order of vanishing of $f$ along $E_i$, $\nu_i-1$ is the order of vanishing of the determinant of the Jacobian of $\mu$ along $E_i$, $E_I=\cap_{i\in S}E_i$ and 
$E_I^\circ=E_I\setminus \cup_{j\in S\setminus I}E_j$. Typically only a few of the denominators $\bL^{\nu_i}T^{-N_i}-1$ survive after cancellations, leading to a formal definition of the notion of pole of $Z^{\mot}_f(T)$, see  \ref{defPoles}.

The poles are largely mysterious and form the subject of the monodromy conjecture of Igusa, Denef and Loeser. In practice, the smaller the resolution is, the smaller the set of pole candidates is, and there is a better chance of understanding the poles. In a few cases, such as plane curves, hyperplane arrangements, and others, a minimal log resolution exists, but in general it does not. 

In higher-dimensional algebraic geometry, the role of minimal log resolutions is played by  minimal models. Setting $\Delta=\sum_{i\in S}E_i$, one runs the minimal model program for $(Y,\Delta)$ over $X=\bA^n$ and achieves a minimal model $(Y',\Delta')$, by Odaka and Xu \cite{OX}. This is only a partial resolution since $Y'$ can have  singularities, but the singularities are mild since $(Y',\Delta')$ is a divisorially log terminal (dlt) pair, see \ref{subNot}. This is in fact a dlt modification of $f$, assuming that $f$ is reduced, see \ref{subMain}.

Xu \cite{Xu} defined a motivic zeta function in terms of dlt modifications of $f$ and raised the question of an intrinsic interpretation for it. However, it turned out that his zeta function 
 depends on the choice of dlt modification if $n\ge 3$, by Nicaise, Potemans, and Veys \cite{NPV}, and hence it cannot have an intrinsic meaning.

We define a birational analog of the motivic zeta function of a reduced polynomial $f$ in terms of a dlt modification of $f$. We show that it is independent of the choice of dlt modification and that it admits an intrinsic meaning in terms of contact loci of arcs. 

\subs{\bf Birational classes.} Let $\Bir^d_\bC$ be the set of birational equivalence classes of complex varieties, always meaning reduced and irreducible in this article, of dimension $d$. Let $\bZ[\Bir^d_\bC]$ be the free abelian group generated by $\Bir^d_\bC$. The group $$\bZ[\Bir_\bC]\vcentcolon=\bigoplus_d \bZ[\Bir^d_\bC]$$ is endowed with a natural graded ring structure, see \cite{KT, NO}. For a complex variety $Z$ of dimension $d$, we denote by $\{Z\}\in\Bir^d_\bC$ its birational equivalence class. If $Z$ is not necessarily irreducible, we denote by $\{Z\}$ the sum of the birational equivalence classes in their respective dimensions of the irreducible components of $Z$. Let $\cL=\{\bA^1\}\in\Bir^1_\bC$.

\subs{\bf Main result.}\label{subMain}  Let $X$ be a smooth complex algebraic variety of dimension $n$ and $D$ a non-zero reduced divisor on $X$. 
Let $\sL_m(X)$
 be the space of $m$-jets on $X$, that is, the space of morphisms of $\bC$-schemes $\Spec\bC[t]/(t^{m+1})\to X$. For $m\ge 1$, define the {\it $m$-contact locus of $(X,D)$} as
$$
\sX_m=\sX_m(X,D)\vcentcolon=\{\gamma\in \sL_m(X)\mid \ord_\gamma(D)=m\}.
$$
The Embedded Nash Problem, asking for a geometric characterization of the irreducible components of $\sX_m$, is still open. A partial answer, which we will use, was given by  \cite{B+} as follows, see also \ref{subMsep}.

A {\it valuation of $(X,D)$} is a divisorial valuation on the function field of $X$ given by the order of vanishing along some prime divisor $E$ on some  birational  modification of $X$, such that the image of $E$ on $X$ is included in $D$. Such a valuation is called a \textit{dlt valuation of \((X,D)\)} if its center is a prime divisor on some projective dlt modification of \((X,D)\). We say that a (dlt) valuation of \((X,D)\) is a \textit{(dlt) \(m\)-valuation} if $\ord_E(D)$ divides $m$. Typically, on a given dlt modification, the centra of 
most   dlt $m$-valuations for  $m\gg 0$ are not divisors. An {\it essential $m$-valuation of $(X,D)$} is an $m$-valuation with center a prime divisor on any $m$-separating log resolution of $(X,D)$,  see \ref{subMsep}.

There exists a natural injection from the set of irreducible components of $\sX_m$ into the set of essential \(m\)-valuations of \((X,D)\), and the Embedded Nash Problem asks to identify its image. By \cite{B+}, the dlt $m$-valuations of $(X,D)$ are also essential \(m\)-valuations and contained in this image, if $X$ is in addition quasi-projective. Denote by $\sX_m^{\dlt}$ the union of the irreducible components of $\sX_m$ corresponding to dlt \(m\)-valuations. 
We now define the {\it birational zeta function of $(X,D)$} as
$$
Z^{\bir}_{X,D}(T) \vcentcolon= \sum_{m\ge 1}\{\sX_m^{\dlt}\}\cL^{-mn}T^m\quad\in\bZ[\Bir_\bC][\cL^{-1}]\llbracket T \rrbracket.
$$

\begin{theorem}\label{thmMain} Let $X$ be a smooth quasi-projective complex algebraic variety and $D$ a non-zero reduced divisor on $X$. Let $\mu:(Y,\Delta)\to (X,D)$ be a  dlt modification of $(X,D)$. Then
$$
Z^{\bir}_{X,D}(T)=\sum_{\emptyset\neq I\subset S}\{E_I\}\prod_{i\in I}\frac{\cL}{\cL^{\nu_i}T^{-N_i}-1},
$$
where $\Delta=\sum_{i\in S}E_i$ is the irreducible decomposition of the reduced pullback of $D$, $N_i=\ord_{E_i}(D)$,   $K_{Y/X}=\sum_{i\in S}(\nu_i-1)E_i$, and $E_I$ for $I\subset S$ is as above.
\end{theorem}

In particular, the right-hand side is independent of the chosen dlt modification, which is not assumed to be projective. Theorem \ref{thmMain} is a birational analog of the formula of Denef and Loeser.

In the body of the paper we prove a more general statement than Theorem \ref{thmMain}, regarding a local birational zeta function that deals with arcs centered on a fixed closed subset $\Sigma$ of  $D$, see Theorem \ref{thmMain2}.

Natural questions are if the analog of the monodromy conjecture holds for the poles of the birational zeta function, and, moreover, if the poles of the birational zeta function are poles of the motivic zeta function, see \ref{defPoles} for the definition of poles.
 
\begin{ques}\label{conj:birmon}(Birational Monodromy Conjecture)
	Let \(X\) be a smooth quasi-projective complex algebraic variety and \(D\)  a non-zero reduced divisor on \(X\). If $s_0\in\bQ$ is a pole of \(Z_{X,D}^{\bir}(T)\),  is then  \(\exp(2\pi is_0)\)  a monodromy eigenvalue of \(D\) at some point of the support of $D$?
\end{ques}
\begin{ques}\label{conj:birmotpoles}
	Let \(X\) be a smooth quasi-projective complex algebraic variety and \(D\)  a non-zero reduced divisor on \(X\). If $s_0\in\bQ$ is a pole of \(Z_{X,D}^{\bir}(T)\), is it then a pole of \(Z_{X,D}^{\mot}(T)\)?
\end{ques}

For plane curves, we show that a local version of Question \ref{conj:birmotpoles} is true. Moreover, in this local case, the set of poles of the topological zeta function, a specialization of the motivic zeta function, is equal to the set of poles of the birational zeta function, see Corollary \ref{cor:topbirpoles}. Combining this with \cite{Lo}, it follows that the local version of Question \ref{conj:birmon} is true for plane curves. In higher dimensions, the converse of Question \ref{conj:birmotpoles} cannot hold, as certain poles of \(Z_{X,D}^{\mot}(T)\) do not even have to appear as candidate poles of \(Z_{X,D}^{\bir}(T)\). See for example \cite[Example 4.5]{Xu}. \\

Any two dlt modifications of $(X,D)$ are crepant-birationally equivalent. One can use the formula on the right-hand side of the equality in Theorem \ref{thmMain} to define a birational zeta function $Z^{\bir}_{X,D,\mu}(T)$ for any dlt resolution $\mu$ of $(X,D)$. A dlt resolution is also a partial log resolution with mild singularities, see \ref{subNot}. The difference with dlt modifications is that a dlt resolution does not have to be a minimal model over $X$. We show more generally  that $Z^{\bir}_{X,D,\mu}(T)$ depends only on the crepant-birational equivalence class of $\mu$, see Proposition \ref{propInv}.  In general, this is different from $Z^{\bir}_{X,D}(T)$ if $\mu$ is not a dlt modification.

\subs{\bf Outline.}
In Section \ref{secProofs}, we prove all the above results, deferring the results on the poles of the birational zeta function for plane curves to Section \ref{secCurv}. 

In Section \ref{secPoles}, we show that the monodromy conjecture cannot hold for birational zeta functions of arbitrary dlt resolutions. We also show that the analog of the former conjecture of Veys on the poles of maximal possible order of the topological zeta function, proven by Nicaise-Xu \cite[Theorem 3.5 (2)]{NX}, holds  for the rational zeta function. The rational zeta function is the specialization of the birational zeta function obtained by sending the birational equivalence class $\{Z\}$ of a variety to $\cL^{\dim Z}$.

In Section \ref{secEx}, we compute birational zeta functions for some concrete examples of pairs \((X,D)\).

\smallskip
\noindent{\bf Acknowledgements.} 
		T. Biesbrouck was supported by the grant 11P4V24N from FWO. N. Budur was supported by a Methusalem grant and the grant G0B3123N from FWO. J. Nicaise was supported by a Methusalem grant and the grant G0B1721N from FWO.
W. Veys was supported by the KU Leuven Grant GYN-E4282-C16/23/010.
We thank Q. Shi for comments. Also, we are grateful to the referee for the careful reading of the manuscript, making valuable suggestions, and pointing out a subtle mistake in the original definition of the local birational zeta function.

\section{Proof of the main result}\label{secProofs}

\subs{\bf Notation and terminology.}\label{subNot}	
We work over $\bC$. A complex algebraic {\it variety} will mean an integral separated finite type $\bC$-scheme. 

A {\it pair} $(Y,\Delta)$ consists of a normal variety $Y$ and a $\bQ$-divisor $\Delta$ such that $K_Y+\Delta$ is $\bQ$-Cartier. We denote by $\Delta^{=1}$ the sum of the prime divisors that have coefficient equal to 1 in $\Delta$.  If all coefficients of $\Delta$ are in $[0,1]$, respectively are $\le 1$, we say $\Delta$ is a {\it boundary}, respectively {\it sub-boundary}. 

We say that $(Y,\Delta)$ is a {\it snc} pair, short for {\it simple normal crossing}, if $Y$ is smooth
and $\Delta$ has simple normal crossing support.

A {\it log resolution}  of a pair $(Y,\Delta)$ is a proper birational morphism $\mu:Y'\to Y$ such that $Y'$ is smooth, and $\mu^{-1}(\Delta)$, the exceptional locus $\Ex(\mu)$ of $\mu$, and their union, are all simple normal crossing divisors. Note that, if $Y$ is in addition $\bQ$-factorial, then $\Ex(\mu)$ is automatically of pure codimension one \cite[VI.1, 1.5]{Ko-rat}.

If $\mu:Y'\to Y$ is a birational morphism of  normal varieties and $(Y,\Delta)$ is a pair, the {\it log pull-back} of $\Delta$ is the $\bQ$-divisor $\Delta_{Y'}$ with $\mu_*(\Delta_{Y'})=\Delta$, defined by
$
K_{Y'}+\Delta_{Y'}\sim_{\bQ} \mu^*(K_Y+\Delta).
$
If $E\subset Y'$ is a prime divisor, the {\it  discrepancy} $a(E,Y,\Delta)$ is the negative of the coefficient of $E$ in $\Delta_{Y'}$.

We say that $(Y,\Delta)$ is a {\it dlt pair}, short for {\it divisorially log terminal}, if it is a pair, $\Delta$ is a boundary, and there exists a closed subset $Z\subset Y$ such that $(Y\setminus Z, \Delta|_{Y\setminus Z})$ is an snc pair and for a log resolution (equivalently, for all log resolutions) $\mu:Y'\to Y$ of $(Y,\Delta)$ with $\mu^{-1}(Z)$ of pure codimension one, the condition $a(E,Y,\Delta)>-1$ is satisfied for every prime divisor $E\subset \mu^{-1}(Z)$.

A {\it dlt modification} of a pair  $(Y,\Delta)$ is a dlt pair $(Y',\Delta')$, together with a proper birational morphism  $\mu:Y'\to Y$, such that $\Delta'=\mu_*^{-1}\Delta+\Ex^1(\mu)$, and $(Y',\Delta')$ is a minimal model over $Y$, that is, $K_{Y'}+\Delta'$ is $\mu$-nef, where by $\Ex^1(\mu)$ we mean the union of the codimension one components of the exceptional locus $\Ex(\mu)$, and by $\mu_*^{-1}(\_)$ we mean taking the strict transform of a divisor.
A {\it dlt resolution of $(Y,\Delta)$} is defined similarly,  but without requiring   $K_{Y'}+\Delta'$  to be $\mu$-nef.

\subs{\bf Birational zeta functions of dlt resolutions.} 
Let $X$ be a smooth complex algebraic variety of dimension $n$, $D$ a non-zero reduced divisor on $X$, and $\Sigma$ a closed subset of the support of $D$.

\begin{definition}\label{defBir} (1)  The {\it birational zeta function of a dlt resolution $\mu:(Y,\Delta)\to (X,D)$} is 
$$
Z^{\bir}_{X,D,\mu}(T)\vcentcolon=\sum_{\emptyset\neq I\subset S}\{E_I\}\prod_{i\in I}\frac{\cL}{\cL^{\nu_i}T^{-N_i}-1}\quad \in \bZ[\Bir_\bC][\cL^{-1}]\llbracket T \rrbracket,
$$
where  $\Supp(\Delta)=\cup_{i\in S}E_i$ is the irreducible decomposition,
\[\ord_{E_i}(D)=N_i,\qquad K_{Y/X}=\sum_{i\in S}(\nu_i-1)E_i,\]
$E_I=\cap_{i\in I}E_i$ for $I\subset S$, and $\{E_I\}\in\bZ[\Bir_\bC]$ is the sum of the birational classes of the irreducible components of $E_I$ in their respective dimensions.  

Note here that, since $X$ is smooth, $\Ex(\mu)$ has pure codimension 1. Moreover, if $\mu$ is a dlt modification, then $\Ex(\mu)\subset \Supp(\mu^*(D))$, equivalently $\Delta=(\mu^*(D))_{\red}$, by the negativity lemma \cite[Lemma 1.17]{Ko}. So in this case $Z^{\bir}_{X,D,\mu}(T)$ agrees with the right-hand side of the equality in Theorem \ref{thmMain}. If \(D\) is the zero locus of a reduced regular function \(f:X\to \bA^1\), we denote $Z^{\bir}_{X,D}(T)$ by $Z^{\bir}_f(T)$.

(2) A `local' version of the birational zeta function of $\mu$ above $\Sigma$ is 
$$
Z^{\bir}_{X,D,\Sigma,\mu}(T)\vcentcolon=\sum_
{\emptyset\neq I\subset S}\{E_I\cap\mu^{-1}(\Sigma)\}_{n-|I|}\prod_{i\in I}\frac{\cL}{\cL^{\nu_i}T^{-N_i}-1}\quad \in \bZ[\Bir_\bC][\cL^{-1}]\llbracket T \rrbracket,
$$
where  $\{V\}_d$ stands for the sum of the birational equivalence classes of the irreducible components of $V_{\red}$ of dimension equal to $d$. Equivalently,
$$
Z^{\bir}_{X,D,\Sigma,\mu}(T)= \sum_{\emptyset\neq I\subset S}\sum_{\substack{Z\subset E_I \\ \mu(Z)\subset\Sigma}} \{Z\}\prod_{i\in I}\frac{\cL}{\cL^{\nu_i}T^{-N_i}-1},
$$
where the second sum  runs over the irreducible components $Z$ of $E_I$ mapping into $\Sigma$.

When $\Sigma$ is the whole support of $D$, this is the previous definition. If \(D\) is the zero locus of a reduced regular function \(f:X\to \bA^1\), we  denote $Z^{\bir}_{X,D,\Sigma}(T)$ by $Z^{\bir}_{f,\Sigma}(T)$.

\end{definition}

For the proof of Proposition \ref{propInv} below we will need a more general definition. 

\begin{definition}\label{rmkDD}
With $X$ and $D$ as above, consider a proper birational morphism $\mu:Y\to X$ and a  $\bQ$-divisor $\Delta$ on $Y$, such that either $(Y,\Delta)$ is a dlt pair, or $\Delta$ is a sub-boundary and $(Y,\Delta)$ is an snc pair. Then the dual complex $\cD(\Delta^{=1})$ is a well-defined regular cell complex by \cite{dFKX}. 
Let $\cup_{i\in S}E_i$  be the irreducible decomposition of the support of $\Delta^{=1}$. 

(1) As in Definition \ref{defBir} (1), to these data we associate a birational zeta function, which we will simply denote by $Z^{\bir}_{\cD(\Delta^{=1})}(T)$. 
We are suppressing from the notation that this also depends on  the birational equivalence classes of the strata corresponding to the faces of $\cD(\Delta^{=1})$, and on the numerical data corresponding to its vertices.

(2) 
Let $\Sigma$ be a  closed subset of the support of $D$. Then as in Definition \ref{defBir} (2), we associate a birational zeta function localized above $\Sigma$, which we will simply denote by $Z^{\bir}_{\cD(\Delta^{=1}),\Sigma}(T)$.

\end{definition}

\begin{definition}
Two dlt resolutions $\mu_j:(Y_j,\Delta_j)\to (X,D)$, $j=1, 2$, are {\it crepant-birationally equivalent} if there exist proper birational morphisms $\pi_j:Y\to Y_j$ such that the log pull-backs of $\Delta_1$ and $\Delta_2$ are equal on $Y$.
\end{definition}

\begin{prop}\label{propInv} Let $X$ be a smooth complex algebraic variety, $D$ a non-zero reduced divisor on $X$, and $\Sigma$ a closed subset of the support of $D$. Then the birational zeta function $Z^{\bir}_{X,D,\Sigma,\mu}(T)$ of a dlt resolution $\mu$ of $(X, D)$ only depends on the crepant-birational equivalence class of  $\mu$.
\end{prop}
\begin{proof}
For the proof, we trace each step in the proof of \cite[Proposition 11]{dFKX} and show that  the birational zeta function  after each step does not change.

Let $\mu_j:(Y_j,\Delta_j)\to (X,D)$ be two dlt crepant-birational resolutions, $j=1, 2.$ By \cite{Sz}, every dlt pair has a log resolution that is an isomorphism over the simple normal crossing locus. Let $\rho_j:Y_j'\to (Y_j,\Delta_j)$ be such a log resolution, and let $\mu'_j=\mu_j\circ\rho_j$. 
Let $\Delta_{Y'_j}$ be the log pull-back of $\Delta_j$ to $Y'_j$. Then $\Delta_{Y'_j}$ is a sub-boundary. 

Moreover, by \cite[9]{dFKX}, the morphism \(\rho_j\) induces a natural bijection between the strata of \(\Delta_{Y_j'}^{=1}\) and \(\Delta_j\), and preserves their birational equivalence classes. Hence, the birational zeta function of $\mu_j$ equals that of $\cD(\Delta_{Y'_j}^{=1})$, where the latter is as in Definition \ref{rmkDD}(1). It is  also clear that the birational zeta function of $\mu_j$ localized above $\Sigma$ equals that of $(\cD(\Delta_{Y'_j}^{=1}),\Sigma)$, where the latter is as in Definition \ref{rmkDD} (2).

Now, by the weak factorization theorem, there is a sequence of blow-ups of smooth centra and their inverses
$$
Y_1'=Z_0 \overset{\pi_0}\dashrightarrow Z_1 \overset{\pi_1}\dashrightarrow\cdots\overset{\pi_{r-1}}\dashrightarrow Z_r=Y'_2,
$$ 
such that there exists an $m$ for which $\pi_0^{-1}\circ\ldots\circ\pi_{i-1}^{-1}:Z_i\dashrightarrow Y_1'$ are morphisms for $i\le m$, and $\pi_{r-1}\circ\ldots\circ\pi_i:Z_i\dashrightarrow Y_2'$ are morphisms for $i\ge m$; for the full formulation of the weak factorization theorem, see \cite[1.7.1]{V}. Let $\Theta_i$ be the log pull-back of $\Delta_{Y_1'}$ for $i\le m$, and the log pull-back of $\Delta_{Y'_2}$ for $i\ge m$, for these morphisms. Since the $\mu_j$ are crepant-birational over $X$, it follows that these two definitions for $\Theta_m$ agree. As in the proof of \cite[Proposition 11]{dFKX}, after each $\pi_i$ the dual complex $\cD(\Theta_i^{=1})$ either does not change, or it changes by a stellar subdivision or its inverse.  If it changes, it corresponds to blowing up a stratum of $\Theta_i^{=1}$, or its inverse. If it does not change, the birational zeta function of $\cD(\Theta_i^{=1})$ does not change either.

So it suffices to show the following. Let $\mu:Y\to (X,D)$ be a log resolution, $\Delta_Y$ a  sub-boundary on $Y$ with support included in $\mu_*^{-1}(D)+\Ex(\mu)$. Let $h_K:Y'\to Y$ be the blow-up of an irreducible component $Z$ of $E_K$ for some $\emptyset\neq K\subset S$, where $\Delta_Y^{=1}=\cup_{i\in S}E_i$. Let $\Delta_{Y'}$ be the log pull-back of $\Delta_Y$ on $Y'$. 
Then we need to show that the birational zeta function of $(\cD(\Delta_Y^{=1}),\Sigma)$ equals that of $(\cD(\Delta_{Y'}^{=1}),\Sigma)$.

For $I\subset S$, let $\sigma_I\subset \cD(\Delta_Y^{=1})$ be the subcomplex spanned by the vertices $i\in I$. The top-dimensional cells of $\sigma_I$ correspond to the irreducible components of the intersection $E_I$. We call them for short {\it the top cells of $\sigma_I$}.

Denote by $\tau_Z$ the top cell of $\sigma_K$ corresponding to $Z\subset E_K$.
Let $E_0'$ be the  exceptional divisor introduced by $h_K$, where we make the convention that $0\not\in S$. Let $S'=S\cup\{0\}$. This is the index set for the irreducible components  of the support of $\Delta_{Y'}^{=1}$, since $\cD(\Delta_{Y'}^{=1})$ is the stellar subdivision of $\cD(\Delta_{Y}^{=1})$ corresponding to an interior point of $\tau_Z$. For $j\in S$ we denote by $E'_j$ the strict transform of $E_j$
 in $Y'$.

We have
\[
Z^{\bir}_{\cD(\Delta_Y^{=1}),\Sigma}(T) = \sum_{K \not\subset I \subset S } \{E_I\cap\mu^{-1}(\Sigma)\}_{n-|I|} \prod_{i \in I} \frac{\cL}{\cL^{\nu_i} T^{-N_i} - 1} + \]
\[+\sum_{K \subset I \subset S} \{E_I\cap\mu^{-1}(\Sigma)\}_{n-|I|} \prod_{i \in I} \frac{\cL}{\cL^{\nu_i} T^{-N_i} - 1}.
\]
In the last sum, for $K\subset I\subset S$, we further decompose

$$
\{E_I\cap\mu^{-1}(\Sigma)\}_{n-|I|} =\sum_{\substack{\tau\text{ top cell of }\sigma_I\\ \mu(Z_\tau)\subset\Sigma \\ \tau\supset \tau_Z}}\{Z_\tau\} +  \sum_{\substack{\tau\text{ top cell of }\sigma_I \\ \mu(Z_\tau)\subset\Sigma \\ \tau\not\supset \tau_Z}}\{Z_\tau\}, 
$$
where  $Z_\tau$ is the  irreducible component of $E_I$ corresponding to $\tau$.

On the other hand,
\[
Z^{\bir}_{\cD(\Delta_{Y'}^{=1}),\Sigma}( T) = \sum_{ K \not\subset I \subset S } \{E_I\cap\mu^{-1}(\Sigma)\}_{n-|I|}\prod_{i\in I} \frac{\cL}{\cL^{\nu_i} T^{-N_i} - 1} +\] 
\[+ \sum_{K \subset I \subset S } \sum_{\substack{\tau\text{ top cell of }\sigma_I\\ \mu(Z_\tau)\subset\Sigma  \\ \tau\not\supset \tau_Z}}\{Z_\tau\} \prod_{i \in I} \frac{\cL}{\cL^{\nu_i} T^{-N_i} - 1} +\]
\[
+ \frac{\cL}{\cL^\nu T^{-N} - 1} \sum_{K \subset I \subset S} \sum_{L \subsetneq K } \sum_{\substack{\tau\text{ top cell of }\sigma_I \\ \tau\supset \tau_Z\\ \mu'(Z'_{\tau,L})\subset\Sigma}}\{Z_\tau\} \cL^{|K| - |L| - 1} \prod_{i \in I \setminus (K \setminus L)} \frac{\cL}{\cL^{\nu_i} T^{-N_i} - 1},
\]
where \( \nu = \sum_{k \in K} \nu_k \) and \( N = \sum_{k \in K} N_k \),   $\mu'=\mu\circ h_K$,
and $Z'_{\tau, L}$ is the irreducible component of $E'_{\{0\}\cup (I\setminus (K\setminus L))}$ corresponding to $\tau$ and $L$. We used in the last term that $\{Z'_{\tau,L}\}=\{Z_\tau\}\cL^{|K|-|L|-1}$.
Also, we have that   $\mu'(Z'_{\tau,L})=Z_\tau$ and hence:
$\mu(Z_{\tau})\subset \Sigma\iff \mu'(Z'_{\tau,L})\subset \Sigma$.

So it suffices to show for every $I$ with $K\subset I\subset S$ that
\[
\prod_{i \in I} \frac{1}{\cL^{\nu_i} T^{-N_i} - 1} = \frac{1}{\cL^\nu T^{-N} - 1} \sum_{L \subsetneq K} \prod_{i \in I\setminus (K\setminus L)} \frac{1}{\cL^{\nu_i} T^{-N_i} - 1}.
\]
This is equivalent to the following equalities:
\[
\prod_{i \in K} \frac{1}{\cL^{\nu_i} T^{-N_i} - 1} = \frac{1}{\cL^\nu T^{-N} - 1} \sum_{L \subsetneq K} \prod_{i \in L} \frac{1}{\cL^{\nu_i} T^{-N_i} - 1} \iff
\]
\[
\cL^\nu T^{-N} - 1 =
\prod_{i \in K} (\cL^{\nu_i} T^{-N_i} - 1) \sum_{L \subsetneq K} \prod_{i \in L} \frac{1}{\cL^{\nu_i} T^{-N_i} - 1} \iff
\]
\[
\cL^\nu T^{-N} - 1 =
 \sum_{L \subsetneq K} \prod_{i \in K\setminus L} (\cL^{\nu_i} T^{-N_i} - 1).
\]
Now we induct on $|K|$. The last equality is trivial for $|K|=1$. For $|K|>1$, fix $k\in K$. Then
\begin{align*}
\sum_{L \subsetneq K} \prod_{i \in K\setminus L} (\cL^{\nu_i} T^{-N_i} - 1)&= (\cL^{\nu_k}T^{-N_k}-1)+((\cL^{\nu_k}T^{-N_k}-1)+1)\sum_{L\subsetneq K\setminus\{k\}}\prod_{i\in (K\setminus\{k\})\setminus L}(\cL^{\nu_i}T^{-N_i}-1) \\
&= \cL^{\nu_k}T^{-N_k} -1+ \cL^{\nu_k}T^{-N_k}\left(\cL^{\sum_{i\in K\setminus\{k\}} \nu_i} T^{-\sum_{i\in K\setminus\{k\}} N_i}-1\right) 
\end{align*}
by the induction assumption. This is now easily seen to equal $\cL^{\nu}T^{-N}-1$. 
\end{proof}

\subs{\bf Contact loci and dlt modifications.}\label{subMsep} We  recall now some facts about contact loci and  how dlt valuations produce irreducible components of the contact loci, following \cite{B+}. 

Let $X$ be a smooth complex algebraic variety of dimension $n$, $D$  a non-zero reduced divisor on $X$, and $\Sigma$  a closed subset of the support of $D$.  
For $m\ge 1$, define the {\it $m$-contact locus of $(X,D,\Sigma)$} as the subset of $m$-jets on $X$ with contact order $m$ along $D$ and center in $\Sigma$:
$$
\sX_m=\sX_m(X,D,\Sigma)\vcentcolon=\{\gamma\in \sL_m(X)\mid \ord_\gamma(D)=m, \gamma(0)\in \Sigma\}.
$$

For $l\in \bZ_{\ge m}\cup\{\infty\}$, we can define similarly the $m$-contact locus $\sX_m^l$ in $\sL_l(X)$ by replacing in the above $\sL_m(X)$ with $\cL_l(X)$. Here,  $\sL_\infty(X)= \varprojlim\cL_m(X)=\Homo_{\bC-\text{sch}}(\Spec(\bC\llbracket t\rrbracket),X)$ is the {\it arc space} of $X$, where the inverse limit is taken for the natural truncation morphisms $\pi_{l, m}:\sL_l(X)\to\sL_m(X)$. While the $\sL_l(X)$ are in general $\bC$-schemes, for us this notation will mean the underlying reduced closed subschemes.

Since $X$ is smooth, the truncation morphism $\pi_{l, m}$ is a locally trivial fibration  with fiber $\bA^{(l-m)n}$, and
$\sX^l_m=\pi_{l,m}^{-1}(\sX_m)$. So the irreducible components of $\sX_m^l$ and those of $\sX_m=\sX_m^m$ determine each other.

Fix an {\it $m$-separating log resolution $\mu:Y\to X$ of $(X,D,\Sigma)$}, that is, $\mu$ is a projective log resolution of $(X,D)$ that is an isomorphism over $X\setminus D$, $\mu^{-1}(\Sigma)$ is pure of codimension 1, and, setting $\mu^*D=\sum_{i\in S}N_i E_i$, for every intersecting $E_i$ and $E_j$ with $i\neq j\in S$, the condition $N_i+N_j>m$ must hold.  Then one has a partition into smooth locally closed subsets
\be\label{eqDec}
\sX^\infty_m = \sqcup_{i\in S_m}\sX^\infty_{m,i},
\ee
where $\sX^\infty_m$ is the $m$-contact locus of $(X,D,\Sigma)$ in the arc space $\sL_\infty(X)$, $\sX^\infty_{m,i}$ is the subset of arcs in $\sX_m^\infty$ which lift, necessarily uniquely, to an arc on $Y$ with center on $E_i^\circ$, and $$S_m=\{i\in S\mid N_i\text{ divides } m \text{ and } \mu(E_i)\subset \Sigma\}.$$ 
The equation (\ref{eqDec}) is still true by replacing  arcs with a high-enough jet level $l\gg m$,  by defining  $\sX^l_{m,i}=\pi_{\infty,l}(\sX^\infty_{m,i})$. Then
$$\sX^l_m = \sqcup_{i\in S_m}\sX^l _{m,i}.
$$

We will denote by $\sX_{m,i}$ the irreducible component of $\sX_m$ determined by the closure of $\sX_{m,i}^l$, assuming the latter is an irreducible component of $\sX_m^l$.

We know `everything' about $\sX^l_{m,i}$ for $l\gg m$ and $i\in S_m$, see \cite[Prop. 3.2]{B+} and its proof. Define 
$$\sY_{m,i}^l\vcentcolon=\{\tilde \gamma \in\sL_l(Y)\mid \ord_{\tilde\gamma}(\mu^*D) =m, \tilde\gamma(0)\in E_i^\circ\}$$ for \(l\gg m\). Then $\mu_l:\sY^l_{m,i}\to \sX^l_{m,i}$ is  a Zariski locally trivial fibration with fiber $\bA^{(\nu_i-1)m/N_i}$. Moreover, the map $\sY^l_{m,i}\to E^\circ_i$, sending a jet to its center, factorizes through the $\bC^*$-normal bundle of $E_i^\circ$, over which it is a Zariski locally trivial fibration with fiber $\bA^{nl-m/N_i}$.

\begin{lemma}\label{lemXmi}
If  $\sX_{m,i}$ is an irreducible component of $\sX_m$ for some $i\in S_m$, then
\[\{\sX_{m,i}\}\cL^{-mn} =\{E_i\} \cL^{-\nu_im/N_i}\cL\quad \in\bZ[{\rm{Bir}}_\bC][\cL^{-1}].\]
\end{lemma}
\begin{proof}
	Since Zariski locally trivial fibrations become trivial in \(\bZ[\Bir_\bC]\), we have for \(l\gg m\) that
	\[\{\sX_{m,i}\}\cL^{-mn} = \{\sX^l_{m,i}\}\cL^{-(l-m)n} \cL^{-mn}= \{E_i\} \cL^{-\nu_im/N_i}\cL.\]
\end{proof}

Now let $\Delta=\mu_*^{-1}(D)+\Ex^1(\mu)$, which in this case equals $(\mu^*(D))_{\red}$. Let 
\be\label{eqMM}
\xymatrix{
(Y,\Delta)  \ar[dr]_\mu \ar@{-->}[rr]^{\phi}& & (Y',\Delta') \ar[dl]^{\mu'} \\
& X &
}
\ee
be a minimal model of $(Y,\Delta)$ over $X$. See \cite[Definition 1.19]{Ko} for definitions. Here $\Delta'=\phi_*(\Delta)=((\mu')^*(D))_{\red}$. The existence of  minimal models is due to \cite{OX}; however, for this we need to assume that $X$ is quasi-projective and $\mu$ is projective, not just proper, which we have already assumed, and in which case $\mu'$ is also projective; see also \cite[Theorem 1.34]{Ko}. Then $(Y',\Delta')$ is a projective (over $X$) dlt modification of $(X,D)$. 

\begin{theorem}\label{thmB} (\cite[Theorem 1.13]{B+}) Let $X$ be a smooth  quasi-projective complex variety, $D$ a  non-zero reduced divisor on $X$, $\Sigma$ a closed subset of the support of $D$, $m\ge 1$, and $\mu$ an $m$-separating log resolution of $(X,D)$. With notation as above, we have the following.
\begin{enumerate}
\item If $E_i$, with $i\in S_m$, does not get contracted on $(Y',\Delta')$, the closure $\ol{\sX^\infty_{m,i}}$ is an irreducible component of $\sX^\infty_m(X,D,\Sigma)$. 

\item
Conversely, every prime divisor $E$ over $\Sigma$ in a projective dlt modification of $(X,D)$ is the strict transform of a prime divisor $E_i$ as in (1) for any $m$-separating log resolution, if
 $\ord_E(D)$ divides $m$.
\end{enumerate}
\end{theorem}

\begin{rmk}
In the proof of \cite[Lemma 3.4]{B+}, which was used for (2) in the above theorem, the last part starting from ``By sufficiently blowing up the klt locus ..." has to be replaced by: ``There exists a projective log resolution $\tilde Y'\to(Y',\Delta')$ which is an isomorphism over the  non-klt locus, by \cite{Sz}. Let $\tilde \Delta'\subset \tilde Y'$ be the reduced inverse image of $D$. So $(Y',\Delta')$ is a minimal model of $(\tilde Y',\tilde \Delta')$ over $X$ by Lemma 2.10. By further blowing up $\tilde Y'$ we obtain an $m$-separating log resolution $Y\to X$ of $(X,D,\Sigma)$. So $v$ is an $m$-valuation, corresponding to the strict transform of $E$ on $Y$. By applying twice Lemma 2.11, which is available for quasi-projective $X$, we conclude that $E$ has a non-zero strict transform on any minimal model of $(Y,\Delta)$ over $X$, where  $\Delta$ is the reduced inverse image of $D$."  
\end{rmk}

\subs{\bf Birational zeta functions and proof of Theorem 1.4.} 

\begin{definition} The set of irreducible components produced by Theorem \ref{thmB} for the contact locus $\sX^\infty_m(X,D,\Sigma)\subset\sL_\infty(X)$ is into one-to-one correspondence with a subset of the irreducible components of $\sX_m=\sX_m(X,D,\Sigma)\subset\sL_m(X)$, since $X$ is smooth, as explained above. Denote by $\sX_m^{\dlt}$ the union of the irreducible components of $\sX_m$ thus produced. 
Define the {\it birational zeta function of $(X,D,\Sigma)$} as
$$
Z^{\bir}_{X,D,\Sigma}(T) \vcentcolon= \sum_{m\ge 1}\{\sX_m^{\dlt}\}\cL^{-mn}T^m\quad\in\bZ[\Bir_\bC][\cL^{-1}]\llbracket T \rrbracket.
$$
If $D$ is the zero locus of a reduced regular function $f:X\to\bA^1$, we will simply denote the birational zeta function by $Z^{\bir}_{f,\Sigma}(T)$. If $\Sigma$ is the support of $D$, we suppress it from the notation and write $Z^{\bir}_{X,D}(T)$ and $Z^{\bir}_f(T)$, respectively.
\end{definition}

Theorem \ref{thmMain} follows from the next theorem by taking $\Sigma$ to be the whole support of $D$. 

\begin{theorem}\label{thmMain2} Let $X$ be a smooth quasi-projective complex algebraic variety, $D$ a non-zero reduced divisor on $X$, and $\Sigma$ a closed subset of the support of $D$. Let $\mu:(Y,\Delta)\to (X,D)$ be a  dlt modification of $(X,D)$. Then 
$$
Z^{\bir}_{X,D,\Sigma}(T)=Z^{\bir}_{X,D,\Sigma,\mu}(T).
$$
\end{theorem}
\begin{proof} 
We will show that the truncations modulo $T^{m+1}$ of these zeta functions, viewed as formal power series in $T$, agree for every $m\ge 1$.

It is known that any two dlt modifications of $(X,D)$ are crepant-birationally equivalent, see the remark following \cite[Definition 15]{dFKX}. Hence, the right-hand side of the equality  is independent of the choice of dlt modification $\mu$ by Proposition \ref{propInv}.  

Fix $m\ge 1$. We change  the notation and let $\mu:(Y,\Delta)\to (X,D)$ now denote an $m$-separating log resolution of $(X,D,\Sigma)$. Let  $(Y',\Delta')$ be a minimal model as in (\ref{eqMM}). So now $\mu'$ is a dlt modification. We show that $Z^{\bir}_{X,D,\Sigma}(T)\equiv Z^{\bir}_{X,D,\Sigma,\mu'}(T)$ modulo $T^{m+1}$. This will finish the proof of the theorem.

By definition, $\mu$ is also $k$-separating for all $1\le k\le m$. Thus, by Theorem \ref{thmB}  applied for this $\mu$ we can compute the truncation of $Z^{\bir}_{X,D,\Sigma}(T)$ modulo $T^{m+1}$:
$$\sum_{k =1}^m\{\sX_k^{\dlt}\}\cL^{-kn}T^k = 
\sum_{k=1}^m  \sum_{\substack{E_i\text{ dlt valuation}\\N_i\mid k, \;\mu(E_i)\subset \Sigma}} 
\{E_i\}\cL \cdot\cL^{-\nu_ik/N_i}T^k = \cL \cdot\sum_{\substack{E_i\text{ dlt valuation}\\N_i\le m, \;\mu(E_i)\subset \Sigma}}  \{E_i\} \sum_{l=1}^{\lfloor m/N_i\rfloor} \cL^{-\nu_il}T^{N_il},
$$
where the first equality follows from Lemma \ref{lemXmi}.

Now we  compute the truncation of $Z^{\bir}_{X,D,\Sigma,\mu'}(T)$ modulo $T^{m+1}$. Let $E_i'=\phi_*E_i$, with \(\phi\) as in (\ref{eqMM}). Let $S'=\{i\in S\mid E_i'\neq 0\}$; this is the index set of the components of $\Delta$ which give dlt valuations. 

We show first that $\mu':Y'\to X$ is an {\it $m$-separating} dlt modification of $(X,D)$, that is,  $N_i+N_j>m$ for any two distinct $E_i'$ and $E_j'$ corresponding to adjacent vertices in the dual complex $\cD(\Delta')$. Suppose for a contradiction that we can pick  $i\neq j\in S'$ and $W$ an irreducible component of $E_i'\cap E_j'$ such that $k\vcentcolon=N_i+N_j\le m$.  
Let $\pi:\tilde Y'\to Y'$ be a log resolution of $(Y',\Delta')$ that is an isomorphism over $Y'\setminus \Delta'$ and such that it factors through the blowup of $Y'$ along $W$. Denote by $E_W$  the strict transform in $\tilde Y'$ of the unique exceptional prime divisor of $\mathrm{Bl}_W Y'\to Y'$ that dominates $W$.
 Then $a(E_W,Y',\Delta')=-1$ and $\ord_{E_W}(D)=k$. 
 
 It suffices to show that $E_W$ defines a dlt valuation of $(X,D)$. Then Theorem \ref{thmB} (with $m=k$) implies that $E_W$ has a divisorial center on every $k$-separating (in particular, every $m$-separating) log resolution of $(X,D)$. This is a contradiction: $E_W$ does not have a divisorial center on $Y$, because $Y\dashrightarrow Y'$ does not have exceptional prime divisors of discrepancy $-1$ with respect to $(Y',\Delta')$.
 
 Let $\tilde\Delta'=(\pi^*\Delta')_{\red}$ and denote by $\pi':(Y'',\Delta'')\to (Y',\Delta')$ a minimal model of $(\tilde Y',\tilde \Delta')$ over $Y'$. Then $\Delta''=((\pi')^*\Delta')_{\red}$ and by \cite[1.35]{Ko} (with $c_i=1$ there), $\Delta''$ consists of the strict transform of $\Delta'$ and of the $\pi$-exceptional prime divisors in $\tilde Y'$ of discrepancy exactly $-1$ along $(Y',\Delta')$, none of which gets contracted on $Y''$.
 In particular, $E_W$ has divisorial center on $Y''$, and $K_{Y''}+\Delta''\sim_\bQ (\pi')^*(K_{Y'}+\Delta')$. Let $\mu''=\mu'\circ\pi'$. Then $\mu'':(Y'',\Delta'')\to (X,D)$ is a dlt modification: indeed, if $C$ is a curve in a fiber of $\mu''$, then $C\cdot(K_{Y''}+\Delta'')=\pi'_*(C)\cdot (K_{Y'}+\Delta')\ge 0$.  Consequently, $E_W$ defines a dlt valuation of $(X,\Delta)$.

Then we have
\begin{align*}
	Z^{\bir}_{X,D,\Sigma,\mu'}(T) &\equiv \sum_{\emptyset\neq I\subset S'} \{E_I'\cap(\mu')^{-1}(\Sigma)\}_{n-|I|}\prod_{i\in I} (\cL\cdot\sum_{k\ge 1}(\cL^{-\nu_i}T^{N_i})^k) \quad \text{ modulo } T^{m+1} \\
	&\equiv \sum_{\emptyset\neq I\subset S'}\{E_I'\cap(\mu')^{-1}(\Sigma)'\}_{n-|I|}\prod_{i\in I} (\cL\cdot\sum_{k= 1}^{\lfloor m/N_i\rfloor}(\cL^{-\nu_i}T^{N_i})^k) \quad \text{ modulo } T^{m+1}.
\end{align*}
Here, if $I$ gives a non-zero summand, the smallest power of $T$ with non-zero coefficient appearing in the product is $T^{\sum_{i\in I}N_i}$.  In this case, if $|I|\ge 2$, there exist $i\neq j\in I\subset S'$ such that $E_i'$ and $E_j'$ intersect non-trivially. Then by the $m$-separating condition for $\mu'$, we have that $N_i+N_j>m$, and hence $I$ does not contribute to the sum modulo $T^{m+1}$.
\end{proof}

The following formula expresses the birational zeta function  in terms of codimension 1 strata. The deeper strata are traded at the expense of a possibly infinite sum.

\begin{prop}\label{propSumE}
With $X$, $D$ and $\Sigma$ as in Theorem \ref{thmMain2}, 
$$Z^{\bir}_{X,D,\Sigma}(T) = 
\sum_{\substack{E \mathrm{\,\,dlt\,\,valuation}\\ \mathrm{centered\,\,in\,\,}\Sigma}}    
\frac{\{E\}\cL}{\cL^{\nu_E}T^{-N_E}-1}.$$
\end{prop}
\begin{proof} 
It is enough to show that the truncations modulo $T^{m+1}$ agree for every $m\ge 1$. Fix such $m$. Fix an $m$-separating log resolution $\mu:Y\to X$ of $(X,D,\Sigma)$. The truncation of the left-hand side modulo $T^{m+1}$ is computed in Theorem \ref{thmMain2} using $\mu$. This is easily seen to agree with the truncation modulo $T^{m+1}$ of the right-hand side, since the dlt valuations of $(X,D)$ centered in $\Sigma$ with $N_E\le m$ have as center a prime divisor on $Y$.
\end{proof}

\begin{remark}
	One can take the formula from Proposition \ref{propSumE} as an alternative intrinsic definition of $Z^{\bir}_{X,D,\Sigma}(T)$ in terms of valuations. The formula in Definition \ref{defBir} can then be deduced from the fact that, given a dlt modification \(\mu:(Y,\Delta)\to(X,D)\), the dlt valuations of $(X,D)$ are precisely the divisorial valuations that are monomial with respect to the snc locus of $(Y,\Delta)$. These correspond canonically to the $\mathbb{Q}$-rational points on the cell complex $\cD(\Delta)$. 
\end{remark}

A similar proof, adjusted to classes in $K_0(\Var_\bC)$, gives a similar result for the motivic zeta function. Let $f:X\to \bA^1$ be a non-constant regular function, not necessarily reduced, on a smooth complex algebraic variety $X$.
Let $\mu:Y\to X$ be a log resolution of $f$ that is an isomorphism above $X\setminus f^{-1}(0)$. Let $\Delta$ be the support of $(f\circ\mu)^{-1}(0)$. Let $\mu_1=\mu$, $\Delta_1=\Delta$. Define inductively $\mu_m$ to be an $m$-separating log resolution of $f$ obtained by blowing up a stratum of $\Delta_{m-1}$, which can always be done by \cite[Proof of Lemma 2.9]{Fl}, 
and let $\Delta_m$ be the support of the inverse image of $\Delta_{m-1}$. Let $\cD(\Delta_\bullet)=\cup_{m\ge 1}\cD(\Delta_m)$ be the limit of refinements of dual complexes $\cD(\Delta_1)\subset\cD(\Delta_{2})\subset \cD(\Delta_3)\subset\ldots$.

\begin{prop} Let $f:X\to \bA^1$ be a non-constant regular function on a smooth complex algebraic variety $X$.
Let $\mu:Y\to X$ be a log resolution of $f$ that is an isomorphism above $X\setminus f^{-1}(0)$ and let $\cD(\Delta_\bullet)$ be as above. Then
$$
Z_f^{\mot}(T) = \sum_{E\in\cD(\Delta_\bullet)}\frac{[E^\circ](\bL-1)}{\bL^{\nu_E}T^{-N_E}-1},
$$ 
where the sum is over the divisorial valuations corresponding to vertices of $\cD(\Delta_\bullet)$, $E^\circ$ is the open stratum of $E$ on any $\mu_m$ on which the center of the valuation is a prime divisor $E$, $\nu_E-1$ is the order of vanishing along $K_{\mu_m}$ for such $\mu_m$, and $N_E=\ord_E(f)$. In particular, the right-hand side is independent of the choices made.
\end{prop}

\subs{\bf Birational nearby cycles and Milnor fibers.}\label{subsBirMil} Recall that a simplified version of the motivic nearby cycles and the motivic Milnor fiber at $x\in f^{-1}(0)$ of a non-constant polynomial $f$ are defined in \cite{DL} by
$$
\Psi^{\mot}_f\vcentcolon=-\lim_{T\to \infty} Z^{\mot}_f(T) = -\sum_{\emptyset\neq I \subset S}[E_I^\circ](1-\bL)^{|I|}, 
$$
$$
\Psi_{f,x}^{\mot}\vcentcolon=-\lim_{T\to \infty} Z^{\mot}_{f,x}(T) = -\sum_{\emptyset\neq I \subset S}[E_I^\circ\cap \mu^{-1}(x)](1-\bL)^{|I|}, 
$$
respectively, in terms of a fixed log resolution $\mu$. Taking the virtual Poincar\'e realization of $\Psi^{\mot}_f$ yields the virtual Poincar\'e realization of the eigenvalue-1 subcomplex for the semisimple monodromy action on the classical nearby cycles complex, endowed with the weight filtration, which is an embedded topological invariant of $f$. The coefficient of the top power of this realization is 
\be\label{eqBirMotC0}
-\sum_{\emptyset\neq I \subset S}(-1)^{|I|}\cdot\#(\text{irreducible components of }E_I) =  \chi(\cD(\Delta)),
\ee
the topological Euler characteristic of the dual complex associated with $\mu$.

Similarly, we can define the {\it birational nearby cycles} and the {\it birational Milnor fiber at $x$} of a reduced polynomial $f$ as
$$\Psi^{\bir}_f\vcentcolon=-\lim_{T\to \infty} Z^{\bir}_f(T) = -\sum_{\emptyset\neq I \subset S'}(-1)^{|I|}\{E_I\}\cL^{|I|},
$$
$$\Psi^{\bir}_{f,x}\vcentcolon=-\lim_{T\to \infty} Z^{\bir}_{f,x}(T) = -\sum_{\emptyset\neq I\subset S'}(-1)^{|I|}\{E'_I\cap(\mu')^{-1}(x)\}_{n-|I|}\cL^{|I|},
$$
respectively,
in terms of a fixed dlt modification $\mu'$ of $f$, the expression being independent of the choice of dlt modification.

We note here the resemblance of $\Psi^{\bir}_f$ with the formulas \cite[(3.2)]{KT} and \cite[(3.3.6)]{NO}, defining the specialization morphism, and the volume morphism, respectively, the central tools in showing specialization of birational types. However, even in the case when $f$ is replaced by a proper morphism, there is one major difference coming from the fact that our formula involves only dlt modifications.

The specialization of $\Psi^{\bir}_f$ under the morphism $\rho:\bZ[\Bir_\bC]\to\bZ[\cL]$ from Definition \ref{defRational} below yields, up to a factor $\cL^n$,
$$
\sum_{\emptyset\neq I \subset S'}(-1)^{|I|}\cdot\#(\text{irreducible components of }E_I) = \chi(\cD(\Delta')),
$$
 the topological Euler characteristic of the dual complex associated with any fixed dlt modification of $f$. By the collapse property of \cite{dFKX}, we can replace $\cD(\Delta')$ by the dual complex $\cD(\Delta)$ of any log resolution of $f$, since the homotopy class does not change, and hence we obtain the same invariant as in (\ref{eqBirMotC0}).

On the other hand,  $\Psi^{\bir}_{f,x}$ could be zero, while $\Psi_{f,x}$ is never zero, see the case $d<n$ of Example \ref{eqCones}. In general we do not yet understand what information $\Psi^{\bir}_{f,x}$ contains about the Milnor fiber.

\subs{\bf Equivariant refinement.} Let $\hat\mu=\varprojlim {\mu_d}$, where $\mu_d$ is the group of $d$-th roots of unity, and let $K_0(\Var_\bC^{\hat\mu})$ be the Grothendieck ring of varieties endowed with a good $\hat\mu$-action from \cite{DL,DLsurvey}.
Define the {\it restricted $m$-contact locus}
$$
\bm{\sX}_m(f)\vcentcolon=\{\gamma\in \sX_m(f)\mid f(\gamma(t))\equiv t^m\text{ mod }t^{m+1}\}.
$$
It is endowed with an obvious $\mu_m$-action. The equivariant refinement of the motivic zeta function of \cite{DL,DLsurvey} is a formal power series in $K_0(\Var_\bC^{\hat\mu})[\bL^{-1}]\llbracket T \rrbracket$, admitting a  formula in terms of any fixed log resolution,
$$
\bm{Z}^{\mot}_{f}(T)\vcentcolon= \sum_{m\ge 1}\{\bm{\sX}_m(f)\}\bL^{-mn}T^m=  \sum_{\emptyset\neq I\subset S}[\tilde E_I^\circ]\prod_{i\in I}\frac{\bL-1}{\bL^{\nu_i}T^{-N_i}-1},
$$
where $\tilde E_I^\circ$ is a canonical unramified  $\mu_{N_I}$-covering associated with the geometry of the log resolution, with $N_I=\mathrm{gcd}(N_i\mid i\in I)$.

Similarly,  our results can be enhanced to the equivariant setting. We state them without proof, indicating only the differences with the above proof.

An equivariant version of the ring generated by birational equivalence classes is defined in \cite[\S 5]{KT}. We define \(\bZ[\Bir_\bC^{\hat\mu}]\) as the equivariant Burnside ring in \cite[\S 5]{KT}, but impose the additional relations \(\{X\times V\}=\{X\times\bA^n\}\), where \(V\) is the \(n\)-dimensional affine space over \(\bC\) with any good \(\hat\mu\)-action, and \(\bA^n\) is taken with the trivial \(\hat\mu\)-action. There is a decomposition $\bm{\sX}_m(f)=\sqcup_{i\in S_m}\bm{\sX}_{m,i}(f)$, similar to (\ref{eqDec}), see \cite{Fl}. 
The only difference is that $\sX_{m,i}(f)$ is not necessarily irreducible; it is however formed by finitely many copies of the same irreducible component translated by the action of $\mu_m$. Theorem \ref{thmB} holds for $\bm{\sX}_m(f)$ too, namely, a prime divisor $E_i$ on an $m$-separating log resolution that corresponds to a dlt $m$-valuation gives rise to (a union of) irreducible components $\overline{\bm{\sX}_{m,i}(f)}$ of $\bm{\sX}_m(f)$. We denote by $\bm{\sX}_m^{\dlt}(f)$ the union of all the irreducible components of $\bm{\sX}_m(f)$ defined this way.
We define the equivariant version of the birational zeta function
$$
\bm{Z}^{\bir}_f(T)\vcentcolon=\sum_{m\ge 1}\{\bm{\sX}_m^{\dlt}(f)\}\cL^{-mn}T^m \in\quad \bZ[\Bir_\bC^{\hat\mu}][\cL^{-1}]\llbracket T\rrbracket.
$$
Then, in terms of a fixed dlt modification of $f$, we have
$$
\bm{Z}^{\bir}_f(T) = \cL^{-1}\cdot\sum_{\emptyset\neq I\subset S}\{\tilde E_I\}\prod_{i\in I}\frac{\cL}{\cL^{\nu_i}T^{-N_i}-1}.
$$
The proof is similar. The only difference is that the analog of the morphism $\sY^l_{m,i}\to E_i^\circ$ for restricted contact loci factors through the  covering $\tilde E_i^\circ$ instead of through the normal $\bC^*$-bundle. In fact, $\tilde E_i^\circ$ can be viewed as the boundary of the normal $S^1$-bundle of $E_i^\circ$. As a  consequence, the analog of  Lemma \ref{lemXmi} becomes: $\{\bm{\sX}_{m,i}(f)\}\cL^{-mn}=\{\tilde E_i\} \cL^{-\nu_im/N_i - 1}$ in $\bZ[\Bir^{\hat\mu}_\bC][\cL^{-1}]$. This accounts for the extra $\cL^{-1}$ in the formula.

As a consequence, one can enhance the birational nearby cycles and Milnor fibers from \ref{subsBirMil} to the equivariant setting.

\section{Poles}\label{secPoles}

We define here the notion of poles of birational zeta functions. This is a thorny issue, like for the motivic zeta functions, see Remark \ref{rmkThorny}, but our choice of definition is enough for our purposes. We verify in this section that the monodromy conjecture cannot hold for birational zeta functions of arbitrary dlt resolutions. We also show that the analog of the former conjecture of Veys on the poles of maximal possible order of the topological zeta function, proven by Nicaise-Xu \cite[Theorem 3.5 (2)]{NX}, holds for the rational zeta function.

\begin{definition}\label{defPoles} 
 Fix $X,D,\Sigma,\mu$ as in Definition \ref{defBir}. By definition of $Z_{X,D,\Sigma,\mu}^{\bir}(T)$ and Proposition \ref{propInv}, there exist subsets \(P\) of \(\bZ_{>0}\times\bZ_{>0}\), minimal with respect to inclusions among those subsets such that
\begin{enumerate}
\item	$Z_{X,D,\Sigma,\mu}^{\bir}(T)$ is in the image of $\bZ[\Bir_\bC]\left[T,\frac{1}{\cL^{a}-T^{b}}\right]_{(a,b)\in P}$ in $\bZ[\Bir_\bC][\cL^{-1}]\llbracket T\rrbracket, \text{ and}$
\item
$P$ is a subset of the set of numerical data $\{(\nu_i,N_i)\mid i\in S\}$	of at least one of the dlt resolutions $\mu'$ in the crepant-birational class of $\mu$. 
\end{enumerate}
A rational number $s_0=-a/b$ for $(a,b)$ in some $P$ as above is called a {\it pole of } $Z_{X,D,\Sigma,\mu}^{\bir}(T)$.

	Applying this definition to dlt modifications $\mu$, we obtain the notion of poles  for the birational zeta function $Z_{X,D,\Sigma}^{\bir}(T)$.
	
	A similar definition is made for motivic zeta functions, by replacing the ring $\bZ[\Bir_\bC]$ with $K_0(\Var_\bC)$.
\end{definition}

\begin{rmk}\label{rmkThorny} (i)  Due to the complicated structure of the ring $\bZ[\Bir_\bC]$ (resp. $K_0(\Var_\bC)$), in particular we do not have a good understanding of its zero-divisors, the following inter-related questions are open regarding the set of poles of a fixed birational (resp. motivic) zeta function.

 Given a finite-sum expression in terms of a fixed dlt (resp. log) resolution, is there only one set $P$ as above obtained after cancellations? Is there a unique set $P$ as above, independent of the choice of dlt  resolution in the same crepant-birational equivalence class (resp. choice of log resolution)? Is the set of poles finite? If one drops the requirement (2) from Definition \ref{defPoles}, do we get the same definition? 
 
These open questions make it difficult, if not impossible, to define the order of a pole in a such a way that it can be computed from any dlt  resolution in the same crepant-birational equivalence class (resp. from any log resolution).
 
 (ii) All these issues disappear for motivic zeta functions after using a homomorphism from $K_0(\Var_\bC)$ to appropriate rings. One general approach is to introduce rational powers of $\bL$ and some further localizations, see the ring $R'$ of \cite[\S 4]{RV}. Another approach is to use any specialization to an integral domain, such as the Hodge-Deligne specialization. Over these rings, one can even define the notion of the order of a pole in a such a way that it can be computed from any log resolution.
 
 For  birational zeta functions, one could build the analog of the ring $R'$ of \cite[\S 4]{RV}. Instead, to keep things simple, we can use a specialization that we define next.
 \end{rmk}

\begin{definition}\label{defRational}
There is an obvious  ring homomorphism
	$$\rho:\bZ[\Bir_\bC]\to\bZ[\cL],$$
sending the birational equivalence class $\{Z\}$ of any variety $Z$ to $\cL^{\dim Z}$.	We also denote by $\rho$  its extension to formal power series in $T$ over these rings. Slightly abusing notation, we define then the \textit{rational zeta function} as $$Z_{X,D}^{\rat}(T)\vcentcolon=\rho(Z_{X,D}^{\bir}(T)).$$ We define $Z_{X,D,\Sigma}^{\rat}(T)$ and $Z_{X,D,\Sigma,\mu}^{\rat}(T)$ similarly.

Since \(\bZ[\cL]\) is an integral domain, the definition of a pole (and its order) for \(Z_{X,D}^{\rat}(T)\) is   straightforward. In particular, if for some root of unity \(\xi\), \(\xi\cL^{-s_0}\) is a pole of \(Z_{X,D}^{\rat}(T)\), then \(s_0\) is a pole of \(Z_{X,D}^{\bir}(T)\).

\end{definition}

\begin{remark}
	The version of Question \ref{conj:birmotpoles} for dlt resolutions which are not dlt modifications fails in general in all dimensions $n>1$. This is because the birational zeta function $Z_{X,D,\Sigma,\mu}^{\bir}(T)$ picks up new poles after performing redundant blow-ups, as the next lemma shows. 

For instance, by first blowing up  a point in some $E_j^\circ$, and then repeatedly blowing up a point in the open part of the new exceptional component, one arrives at the setting of the lemma, with moreover $r=1$.
\end{remark}

\begin{lemma}\label{lem:extremalpointpole}
	Let \(X\) be a smooth quasi-projective complex algebraic variety of dimension \(n\), \(D\) a non-zero reduced divisor on \(X\), and \(\Sigma\) a closed subset of the support of \(D\). Let  \(\mu:(Y,\Delta)\to(X,D)\) be a dlt resolution such that \(\mu^{-1}(\Sigma)\) contains an exceptional prime divisor \(E_0\) satisfying
	\begin{enumerate}[label=(\arabic*)]
		\item \(\nu_0/N_0\neq\nu_i/N_i\) for all other irreducible components \(E_i\) of $\Delta$, and
		\item \(E_0\) intersects exactly \(r\) irreducible components \(E_1,\ldots,E_r\) of \(\Delta\), and \(E_{I\cup\{0\}}\) is non-empty and irreducible for every \(I\subset\{1,\ldots,r\}\). That is, \(E_0\) corresponds to an outer vertex of an outer \(r\)-simplex in \(\mathcal{D}(\Delta)\).
	\end{enumerate}
	Then $-\nu_0/N_0$ is a pole of \(Z_{X,D,\Sigma,\mu}^{\bir}(T)\).
\end{lemma}
\begin{proof}
	It suffices to show that \(\cL^{-\nu_0/N_0}\) is a pole of \(Z_{X,D,\Sigma,\mu}^{\rat}(T)\). The contribution of the strata contained in \(E_0\) to \(Z_{X,D,\Sigma,\mu}^{\rat}(T)\) is given by
	$$\frac{\cL^n}{\cL^{\nu_0}T^{-N_0}-1}\sum_{I\subset\{1,\ldots,r\}}\frac{1}{\prod_{i \in I}(\cL^{\nu_i}T^{-N_i}-1)}
$$
	$$		=\frac{\cL^n\sum_{I\subset\{1,\ldots,r\}}\prod_{i \notin I}(\cL^{\nu_i}T^{-N_i}-1)}{(\cL^{\nu_0}T^{-N_0}-1)\prod_{j \in \{1,\ldots,r\}}(\cL^{\nu_j}T^{-N_j}-1)} =\frac{\cL^{n+\sum_j\nu_j}T^{-\sum_jN_j}}{(\cL^{\nu_0}T^{-N_0}-1)\prod_{j \in \{1,\ldots,r\}}(\cL^{\nu_j}T^{-N_j}-1)},
	$$	where the last equality  follows by induction on \(r\), as shown in the last part of the proof of Proposition \ref{propInv}. Clearly, the candidate pole \(\cL^{-{\nu_0}/{N_0}}\) cannot cancel, so it is a pole.
\end{proof}
The following proposition is an analog for birational zeta functions of Veys's conjecture for topological zeta functions in \cite[(0.2)]{LV}, which was proven in \cite[Theorem 3.5]{NX}. We give a completely similar argument.
Recall that $\lct_x(X,D)$ is the minimum value of $\nu_i/N_i$ with $x\in\mu(E_i)$, in the notation of Definition \ref{defBir} (1), where in addition $\mu$ is assumed to be a log resolution of $(X,D)$.

\begin{proposition}\label{prop:lctpole}
	Let \(f:X\to\bA^1\) be a non-constant regular function on a smooth quasi-projective complex algebraic variety of dimension \(n\), defining a reduced divisor \(D=f^{-1}(0)\), and let \(x\) be a closed point in the support of \(D\). Let \(\mu:(Y,\Delta)\to(X,D)\) be a dlt resolution for which we retain the notation from Definition \ref{defBir}. Let 
	$$
	\lct_{x,\mu}^{\bir}(X,D)\vcentcolon=\min\left\{ \frac{\nu_i}{N_i} \,\,\middle|\,\, \exists \; J\subset S\text{ with }i\in J \text{ and } \dim (E_J\cap\mu^{-1}(x))=n-|I|\ge 0\right\},
	$$
	assuming that the set on the right is non-empty. Denote by \(m\) the largest positive integer such that there exists a subset \(J\) of \(S\) of cardinality \(m\) with $\dim (E_J\cap\mu^{-1}(x))=n-m\ge 0$, and \({\nu_j}/{N_j}=\lct_{x,\mu}^{\bir}(X,D)\) for every \(j\) in \(J\). Then the following properties hold.
	\begin{enumerate}[label=(\roman*)]
		\item The  rational zeta function \(Z_{f,x,\mu}^{\rat}(T)\) has a pole of order \(m\) at  \(T=\cL^{-s_0}\), with \(s_0=-\lct_{x,\mu}^{\bir}(X,D)\).  Moreover, this is the largest possible value 
of  \(s_0\) such that \(Z_{f,x,\mu}^{\rat}(T)\) has a pole  at  \(\cL^{-s_0}\).

 In particular, 	the	birational zeta function \(Z_{f,x,\mu}^{\bir}(T)\) has a pole at \(s_0=-\lct_{x,\mu}^{\bir}(X,D)\) if  $\mu^{-1}(x)$ has pure codimension one.

		\item Conversely, assume $\mu$ is an isomorphism over $X\setminus D$ (this is the case for example if $\mu$ is a dlt modification of $(X,D)$). If \(T=\cL^{-s_0}\) is a pole of order \(n\) of \(Z_{f,x,\mu}^{\rat}(T)\), then \(s_0= -\lct_{x,\mu}^{\bir}(X,D)=-\lct_x(X,D)\), where $\lct_x(X,D)$ is the log canonical threshold around $x$, and \(m=n\). Moreover, \(s_0\) is of the form \(-1/N\) for some positive integer \(N\).
	\end{enumerate}
\end{proposition}
\begin{proof}
		(i)  
By the definition of log canonical threshold and the explicit formula in Definition \ref{defBir}, it is clear that \(s_0=-\lct_{x,\mu}^{\bir}(X,D)\) is the largest possible value of \(s_0\) inducing a pole at \(T=\cL^{-s_0}\) for \(Z_{f,x,\mu}^{\rat}(T)\). Similarly, it is clear by construction that the order of \(\cL^{-s_0}\) as a pole is at most \(m\). To see that \(\cL^{-s_0}\), with \(s_0=-\lct_{x,\mu}^{\bir}(X,D)\), is actually a pole of order \(m\), let \(E_I\) be a closed stratum with \(I=\{1,\ldots,r\}\) and $\dim (E_I\cap\mu^{-1}(x))=n-r$, where \(r\geq m\) and \(\nu_i/N_i=\lct_{x,\mu}^{\bir}(X,D)\) for every \(1\leq i\leq m\). By assumption, such an \(I\) certainly exists. Multiplying the term of \(Z_{f,x,\mu}^{\rat}(T)\) induced by \(E_I\) with \((\cL^{-s_0}T^{-1}-1)^m\), and evaluating in \(T=\cL^{-s_0}\), then yields
		\[\frac{\cL^{n}}{\prod_{1\leq j\leq m}N_j\prod_{m+1\leq i\leq r}(\cL^{\alpha_i}-1)},\]
		with \(\alpha_i=\nu_i+s_0N_i\). Since \(\alpha_i>0\) for all \(i\in\{m+1,\ldots,r\}\), one easily verifies that such terms can never add up to \(0\), so \(\cL^{-s_0}\) is a pole of order \(m\) of \(Z_{f,x,\mu}^{\rat}(T)\).

	(ii)	 If \(\cL^{-s_0}\) is a pole of order \(n\) of \(Z_{f,x,\mu}^{\rat}(T)\), then it follows from the explicit formula in Definition \ref{defBir} that there must exist a subset \(J\) of \(S\) of cardinality \(n\) with \(E_J\cap\mu^{-1}(x)\neq\emptyset\) necessarily of dimension zero, and \(s_0=-\nu_j/N_j\) for every \(j\in J\). By \cite{Sz}, the dlt pair $(Y,\Delta)$ has a log resolution that is an isomorphism over the simple normal crossing locus. In particular it is an isomorphism over a neighborhood of a point in $E_J\cap\mu^{-1}(x)$. Using this log resolution, we are in the setup of \cite[Theorem 2.4]{NX}, which requires in particular that $\mu$ is an isomorphism over $X\setminus D$, and 
	which implies that \(s_0=-\lct_x(X,D)\), and thus \(m=n\). It was already shown in \cite{LV} that \(s_0\) must now be of the form \(-1/N\). Since $\lct_x(X,D)\le \lct_{x,\mu}^{\bir}(X,D)$ and the negative of the latter is the maximal pole, the two log canonical thresholds are equal.
\end{proof}

\begin{remark}\label{rem:lctpole}
	If $\mu$ is a dlt modification, $\lct_x(X,D)$ is computed by some divisor in $\Delta$,  
	for example by \cite[Theorem 1.16]{B+}.
	In particular, \(s_0=-\lct_x(X,D)\) (resp. \(\cL^{-s_0}\)), is always a pole of \(Z_{f,x}^{\bir}(T)\) (resp. \(Z_{f,x}^{\rat}(T)\)) if  $\lct_x(X,D)=\lct_{x,\mu}^{\bir}(X,D)$. This happens for example if $\mu^{-1}(x)$ has pure codimension one. Note that $\lct_{x,\mu}^{\bir}(X,D)$ is independent of the dlt modification $\mu$ by the main result of this article and the previous proposition.
\end{remark}

\section{Examples}\label{secEx}

\subs{\bf Cones over smooth projective hypersurfaces.}\label{eqCones}  Let $n\ge 3$, $d\ge 1$.
Consider $X=\bA^n$ and $D=f^{-1}(0)$, where $f\in\bC[x_1,\ldots,x_n]$ is a reduced homogeneous polynomial of degree $d$ such that its zero locus  $H\subset \bP^{n-1}$ is smooth. Then $D$ has an isolated singularity at the origin. Let $\mu:Y\to X$ be the blowup at the origin, $\Delta=(\mu^*D)_{\red}=\tilde D+E$, where $\tilde D$ is the strict transform of $D$ and $E$ is the exceptional divisor. This is a $d$-separating log resolution of $(X,D)$, and of $(X,D,\{0\})$. The dual complex $\cD(\Delta)$ consists of two vertices connected by an edge. For every $m\ge 1$, there is a unique minimal way of blowing up inductively further to obtain an $m$-separating log resolution $\mu_m:(Y_m,\Delta_m=(\mu_m^*D)_{\red})\to X$, with $\mu=\mu_1=\ldots=\mu_d$, whose dual complex $\cD(\Delta_m)$ is a chain obtained from $\cD(\Delta)$ by inserting vertices.

If $d<n$, then $(X,D)$ is dlt, so it is a dlt modification of itself. Using it, we compute $$Z^{\bir}_{f}(T)=\frac{\cL^2\{H\}}{\cL T^{-1}-1}.$$ 
Now let us  compute it  using $\sX_m\vcentcolon=\sX_m(X,D)$. For every $m\ge 1$, the valuation given by $D$ on $X$ is the only dlt $m$-valuation of $(X,D)$, since $(X,D)$ is a minimal model of $(Y_m,\Delta_m)$ over $X$ by \cite[1.27]{Ko}. Hence,
$
\sum_{m\ge 1}\{\sX_m^{\dlt}\}\cL^{-mn}T^m = \sum_{m\ge 1}\{D\}\cL\cdot\cL^{-m}T^m=\sum_{m\ge 1}\{H\}\cL^2\cdot\cL^{-m}T^m$, 
using Lemma \ref{lemXmi}, which indeed equals $\cL^2\{H\}/(\cL T^{-1}-1)$.

If $d\ge n$, then $(Y_m,\Delta_m)$ is  a minimal model of itself over $X$, so it is a dlt modification of $(X,D)$, see \cite[5.11]{PH}. In particular, we can use the simplest one from our collection of dlt modifications $\mu_m$, namely $\mu$, to compute 
$$Z^{\bir}_{f}(T)=\cL^n\cdot \frac{\{H\}\cL^2T^{-d}+\cL T^{-1}-1}{(\cL^nT^{-d}-1)(\cL T^{-1}-1)}.$$
Alternatively, we compute $Z^{\bir}_{f}(T)$  by making $\sX_m^{\dlt}$ explicit, using  \cite{PH}. The dlt $m$-valuations of $\sX_m$ are given by $S_m$. There is a bijection between $S_m$ and $[-m/d,0]\cap\bZ$, see \cite[Proof of Theorem 2.11]{PH}. If $i\in[-m/d,0]\cap\bZ$, denote by $E_i$ the respective component of $\Delta_m$. Then $E_0=\tilde D$,  $E=E_{- m/d}$ if $d$ divides $m$, and $m\nu_i/N_i=m+i(d-n)$. Note that $\{E_i\}$ is $\cL\{H\}$ if $i\neq -m/d$, and it equals $\cL^{n-1}$ if $i=-m/d$. Thus, using Lemma \ref{lemXmi},
$$
\{\sX^{\dlt}_m\}\cL^{-mn} =\left\{
\begin{array}{ll}
\sum_{i\in (-m/d,0]\cap\bZ}\cL^2\{H\}\cL^{-m-i(d-n)} & \text{ if }d\nmid m,\\
 \cL^{n-mn/d}+\sum_{i\in(-m/d,0]\cap\bZ}\cL^2\{H\}\cL^{-m-i(d-n)} & \text{ if }d|m.
\end{array}
\right. 
$$
One checks  after a lengthy computation that $\sum_{m\ge 1}\{\sX^{\dlt}_m\}\cL^{-mn}T^m$ is indeed equal to the above value of $Z^{\bir}_f(T)$.

The above birational zeta functions correspond to taking $\Sigma=D$. If we take $\Sigma=\{0\}$, the singular point of $D$, then $Z^{\bir}_{f,0}(T)=0$ for $d<n$, and $Z^{\bir}_{f,0}(T)=Z_f^{\bir}(T) - \cL\{H\}/(\cL T^{-1}-1)$ for $d\ge n$.

\subs{\bf Log canonical case.}\label{exLc} Let $X$ be a smooth quasi-projective variety  and $D\subset X$ a reduced divisor such that $(X,D)$ is log canonical. By \cite[Theorem 1.21 (iii)]{B+}, $\sX_m(X,D) =\sX_m^{\dlt}(X,D)$. Hence, $Z^{\bir}_{X,D}(T)=\sum_{m\ge 1}\{\sX_m\}\cL^{-mn}T^m$. In terms of a fixed $m$-separating log resolution $\mu:(Y,\Delta)\to (X,D)$, the irreducible components of $\sX_m(X,D)$ are given by the $E_i$ with $i\in S_m$ such that $\nu_i=N_i$. The set of such $E_i$, or rather the set of divisorial valuations defined by them, is intrinsic to $(X,D)$, let us denote it by $\Sk_m=\Sk_m(X,D)$. Starting  from any log resolution $Y\to (X,D)$, one determines the {\it essential skeleton} $\Sk=\Sk(X,D)\vcentcolon=\cup_{m\ge 1}\Sk_m$  as the prime divisors introduced by blowing up inductively the strata of $\Delta_Y^{=1}$, where $\Delta_Y$ is the log pull-back of $D$. We have by Proposition \ref{propSumE} that
$$
Z^{\bir}_{X,D}(T) = \sum_{E\in\Sk}\frac{\{E\}\cL}{(\cL T^{-1})^{N_E}-1} =  \sum_{\emptyset\neq I\subset \Sk_m}\{E_I\}\prod_{i\in I}\frac{\cL}{(\cL T^{-1})^{N_i}-1}
$$
for every $m\ge 1$.

If, in addition, $D$ is irreducible and has rational singularities, then $\Sk_m$ consist only of the valuation given by $D$, for all $m\ge 1$, by \cite[Theorem 1.21]{B+}. In this case $
Z^{\bir}_{X,D}(T) = {\{D\}\cL}/(\cL T^{-1}-1).
$

The above birational zeta functions correspond to taking $\Sigma$ equal to the  support of $D$. For an arbitrary closed subset $\Sigma$ of the support of $D$, it is not necessarily true that $\sX_m(X,D,\Sigma)=\sX_m^{\dlt}(X,D,\Sigma)$. Using \cite[Theorem 1.21 (i), (ii)]{B+}, there is nevertheless a similar formula:
$$
Z^{\bir}_{X,D,\Sigma}(T)=\sum_{E\in\Sk(\Sigma)}\frac{\{E\}\cL}{(\cL T^{-1})^{N_E}-1}=  \sum_{\substack{\emptyset\neq I\subset \Sk_m \\ I\cap \Sk_m(\Sigma)\neq\emptyset}}
\{E_I\}\prod_{i\in I}\frac{\cL}{(\cL T^{-1})^{N_i}-1},
$$
for every $m\ge 1$,
where $\Sk_m(\Sigma)$ consists of the valuations given by those $E_i$ in $\Sk_m$ with $\mu(E_i)\subset\Sigma$, for a fixed $m$-separating log resolution $\mu$ of $(X,D,\Sigma)$ such that all closed strata $E_I$ are irreducible, and $\Sk(\Sigma)=\cup_{m\ge 1}\Sk_m(\Sigma)$.

\subs{\bf Hyperplane arrangements.} Let $(X=\bA^n,D)$ be a reduced hyperplane arrangement. Then $\sX_m\vcentcolon=\sX_m(X,D)=\sX_m^{\dlt}(X,D)$ by \cite[Theorem 1.14]{B+}. Hence, $Z^{\bir}_{X,D}(T)=\sum_{m\ge 1}\{\sX_m\}\cL^{-mn}T^m$ in this case as well. 

An {\it edge} is any intersection of hyperplanes in $D$. Let $L(D)$ be the set of edges other than the ambient space $X$. The canonical log resolution of $\mu:(Y,\Delta)\to(X,D)$, obtained by blowing up successively in increasing dimension the strict transforms of the edges of $D$,  is a dlt modification, where $\Delta=(\mu^*(D))_{\red}$, by \cite[Proposition 4.4]{B+}. Hence, it can be used to compute that $$Z^{\bir}_{X,D}(T)=\cL^n\cdot\sum_{\emptyset\neq\cF\subset L(D)}\prod_{Z\in \cF}\frac{1}{\cL^{\nu_Z}T^{N_Z}-1},$$
where the sum is over the nested subsets $\cF$ of edges, $\nu_Z$ is the codimension of $Z$ in $X$, and $N_Z$ is the number of hyperplanes in $D$ containing $Z$. 

Recall that there is also a more economical log resolution of $(X,D)$, obtained by blowing up only the dense edges. We have not shown that it gives a dlt modification. However, taking a minimal model of it over $X$, and using the resulting dlt modification to compute $Z^{\bir}_{X,D}(T)$, we get that the birational monodromy conjecture holds, since all the candidate poles from this more economical log resolution already give monodromy eigenvalues by \cite{BMT}.

\section{Plane curves}\label{secCurv}

In this section, we give a characterisation of the poles of birational zeta functions for plane curves and make a comparison with the associated topological zeta function. Throughout this section, let \(X\) be a Zariski open subset of \(\bA^2=\Spec\bC[x,y]\),
 \(f:X\to\bA^1\)  a fixed non-constant morphism defining a non-zero reduced divisor \(D=f^{-1}(0)\) on \(X\), and \(a\in\Supp(D)\)  a fixed closed point. We assume that the germ of \(f\) at \(a\) does not already have normal crossing. That is, \((f,a)\) is not analytically isomorphic to \((x,0)\) or \((xy,0)\), in which case the local birational zeta functions at \(a\) are given by \(0\) and \(\cL^2/(\cL T^{-1}-1)^2\), respectively.

We will consider the  {\it local topological zeta function} of $f$ at $a$,  $$Z_{f,a}^{\top}(s)\vcentcolon=\sum_{\emptyset\neq I\subset S}\chi\left(E_I^\circ\cap \mu^{-1}(a)\right)\prod_{i\in I}\frac{1}{N_is+\nu_i}\quad\in\bC(s),$$
defined in terms of a log resolution $\mu:Y\to X$ of $f$ that is an isomorphism over $X\setminus D$, where $\chi(\_)$ is the topological Euler characteristic, and the rest of the notation is as in \ref{subsMot}. The rational function $Z^{\top}_{f,a}(s)$ is a certain specialization of the local version of the motivic zeta function and it is independent of the choice of log resolution $\mu$, see \cite{DL}. 

In the case of plane curves, there exists a unique minimal log resolution which we denote by
 \(\mu _{\min}:(Y _{\min},\Delta_{\min})\to (X,D)\), with \(\Delta_{\min}=(\mu^*_{\min}(D))_{\red}\).  There is also a unique minimal dlt modification $\mu_{\dlt}:(Y_{\dlt},\Delta_{\dlt}) \to (X,D)$, with \(\Delta_{\dlt}=(\mu^*_{\dlt}(D))_{\red}\), where the latter is obtained as the minimal model of $(Y_{\min},\Delta_{\min})$ over $X$. In terms of the geometry of $\Delta_{\min}$ and following the terminology from \cite[Proposition 6.10]{B+},
 \(\mu_{\dlt}\) can be obtained from \(\mu_{\min}\) by contracting all \textit{maximally admissible twigs} of \(\Delta_{\min}\), that is,  maximal chains of exceptional \(\bP^1\)'s, say \(T_1,\ldots,T_n\), satisfying \(T_1\cdot(\Delta_{\min}-T_1)=1\) and \(T_i\cdot(\Delta_{\min}-T_i)=2\) for all \(i\in\{2,\ldots,n\}\), and ordered in such a way that \(T_i\cdot T_{i+1}=1\) and \(T_i\cdot T_j=0\) for \(|i-j|>1\). For the topological zeta function, we already have the following characterisation of the poles of \(Z_{f,a}^{\top}(s)\) in terms of the geometry of \(\Delta_{\min}\).
 \begin{theorem}\label{thm:polestop}(\cite[Theorem 4.3]{V95})
 	We have that \(s_0\) is a pole of \(Z_{f,a}^{\top}(s)\) if and only if either \(s_0=-\nu_i/N_i\) for some exceptional curve \(E_i\) in \(Y _{\min}\) intersecting at least 3 times other components, or \(s_0=-1/N_i\) for some irreducible component \(E_i\) in \(Y _{\min}\) of the strict transform of \(D\).
 \end{theorem} 
 The goal of this section is to prove a similar characterisation of the poles of \(Z^{\bir}_{f,a}(T)\) in terms of the geometry of $\Delta_{\dlt}$, leading to a comparison result between the poles of \(Z^{\bir}_{f,a}(T)\) and of \(Z^{\top}_{f,a}(s)\). In particular, this gives a positive answer to the local birational monodromy conjecture for plane curves. 
 
Before starting the proof, we fix some assumptions and notation. By possibly shrinking $X$ to a Zariski open subset, we can assume that $a$ is the only singular point of $D$ and that the exceptional loci of $\mu_{\min}$ and $\mu_{\dlt}$ are completely over $a$. We will use $\mu_{\dlt}$ to compute \(Z^{\bir}_{f,a}(T)\) via Definition \ref{defBir}(2) and Theorem \ref{thmMain2}.

The contribution to \(Z^{\bir}_{f,a}(T)\) of the strata contained in an exceptional curve \(E_{\exc}\) in \(\mu_{\dlt}^{-1}(a)\) with numerical data \((\nu,N)\), having exactly \(r\) intersections with other components of $\mu^{-1}_{\dlt}(D)$, say \(E_1,\ldots,E_r\) (not necessarily distinct), is then given by
\begin{equation*}
	\mathrm{Contr}_a(E_{\exc})=\frac{\cL^2}{\cL^{\nu}T^{-N}-1}\left(1+\sum_{i=1}^{r}\frac{1}{\cL^{\nu_i}T^{-N_i}-1}\right).
\end{equation*}
Assuming that \({\nu}/{N}\neq{\nu_i}/{N_i}\) for all \(i\in\{1,\ldots,r\}\), and letting \(\xi\) be an \(N\)-th root of unity, we define the \textit{normalized residue \(\mathcal{R}_a(E_{\exc},\xi)\) of \(\mathrm{Contr}_a(E_{\exc})\) at \(\xi\cL^{{\nu}/{N}}\)} by evaluating \(\mathrm{Contr}_a(E_{\exc})\cdot\cL^{-2}(\xi\cL^{{\nu}/{N}}T^{-1}-1)\) in \(T=\xi\cL^{{\nu}/{N}}\). When \(\xi=1\), this gives
\begin{equation}\label{eqRes}
	\mathcal{R}_a(E_{\exc})\vcentcolon=\mathcal{R}_a(E_{\exc},1)=\frac{1}{N}\left(1+\sum_{i=1}^{r}\frac{1}{\cL^{\alpha_i}-1}\right),
\end{equation}
where \(\alpha_i\vcentcolon=\nu_i-{\nu}{N_i}/N\) for all \(i\in\{1,\ldots,r\}\). 

Next, an analytically irreducible component \(E_{\str}\) of the strict transform of \(D\) under $\mu_{\dlt}$ has numerical data \((\nu,N)=(1,1)\) and has exactly one intersection with an irreducible component from \(\mu_{\dlt}^{-1}(a)\), say \(E_1\). Since \(\mu(E_{\str})\not\subset\{a\}\), this intersection point yields the only contribution to \(Z_{f,a}^\bir(T)\) of the strata contained in \(E_{\str}\). Hence this contribution and its (similarly defined) normalized residue are
\begin{equation}\label{eqResStrict}
	\mathrm{Contr}_a(E_{\str})=\frac{\cL^2}{(\cL T^{-1}-1)(\cL^{\nu_1}T^{-N_1}-1)}\qquad\text{and}\qquad \mathcal{R}_a(E_{\str})=\frac{1}{\cL^{\alpha_1}-1}.
\end{equation}

\begin{remark}\label{rem:simplifypole}
	Note that \(Z^{\bir}_{f,a}(T)=Z^{\rat}_{f,a}(T)\) belongs to the field of fractions of \(\bC[\cL,T]\) here. Hence, \(s_0\in\bQ\) being a pole of \(Z^{\bir}_{f,a}(T)\) in the sense of Definition \ref{defPoles} is equivalent to \(\xi\cL^{-s_0}\) being a pole of \(Z^{\bir}_{f,a}(T)\) in the classical sense of the word for some root of unity \(\xi\in\bC\) of finite order, that is, \(\xi\cL^{-s_0}T^{-1}-1\) appears in the denominator of \(Z^{\bir}_{f,a}(T)\) after all possible cancellations as a rational function in \(T\) over the \(\bC\)-algebra \(\bC[\cL, \cL^{\nu_i/N_i}\mid i\in S]\). In case \(\xi\cL^{-s_0}\) is such a candidate classical pole of order 1, it is an actual classical pole if and only if all normalized residues \(\mathcal{R}_a(E,\xi)\) defined above with \(\nu/N=s_0\) do not add up to \(0\).
\end{remark}

Next, we recall a classical result concerning the numerical data, first proven for arbitrary plane curves by Loeser \cite{Lo} (preceded by some partial results by Strauss, Meuser and Igusa). A shorter and more conceptual proof for parts (i)-(iii) can be found in \cite[Lemma 4.1]{V25}.
\begin{lemma}\label{lem:numdata}
	Let \(E_{\exc}\) be an exceptional curve in \(Y _{\min}\) intersecting exactly \(r'\) times other components, say \(E_1,\ldots,E_{r'}\). Denote \(\kappa\vcentcolon=-E_{\exc}^2\), the negative of the self-intersection number of \(E_{\exc}\) on  \(Y _{\min}\), and let \(\alpha_i\vcentcolon=\nu_i-{\nu}{N_i}/N\) for every \(i\in\{1,\ldots,r'\}\). Then
	\begin{enumerate}[label=(\roman*)]
		\item \(\kappa N=\sum_{i=1}^{r'}N_i\);
		\item \(\kappa\nu=\sum_{i=1}^{r'}(\nu_i-1)+2\);
		\item \(\sum_{i=1}^{r'}(\alpha_i-1)+2=0\);
		\item \(-1\leq\alpha_i<1\) for every \(i\in\{1,\ldots,r'\}\).
	\end{enumerate}
\end{lemma}

In the following main result of this section, we call a point on an irreducible component of \(\Delta_{\dlt}\) \textit{special} when it is either an intersection point with another irreducible component of \(\Delta_{\dlt}\) or a singular point of \(Y _{\dlt}\). Note that the number of special points on such an irreducible component equals the number of intersection points of its strict transform in \(\Delta_{\min}\).
\begin{theorem}\label{thm:polescurves}
	We have that \(s_0\) is a pole of \(Z_{f,a}^{\bir}(T)\) if and only if either \(s_0=-\nu_i/N_i\) for some exceptional curve \(E_i\) in \(Y_{\dlt}\) containing at least 3 special points, or \(s_0=-1\).
\end{theorem}
\begin{proof}
	By Remark \ref{rem:simplifypole}, \(s_0\in\bQ\) is a pole of \(Z^{\bir}_{f,a}(T)\) in the sense of Definition \ref{defPoles} if and only if \(\xi\cL^{-s_0}\) is a \lq classical pole\rq\ of \(Z_{f,a}^{\bir}(T)\) for some root of unity \(\xi\in\bC\) of finite order. Recall that by Theorem \ref{thmMain2} all such candidate classical poles of \(Z_{f,a}^{\bir}(T)\) are of the form \(\xi\cL^{-s_0}\) with \(s_0=-\nu_i/N_i\) and \(\xi\) an \(N_i\)-th root of unity for some irreducible component \(E_i\) of \(\Delta_{\dlt}\).
	
	First suppose that such an \(E_i\) intersects another  irreducible component \(E_j\) of \(\Delta_{\dlt}\) with \(\nu_i/N_i=\nu_j/N_j\). Then \(\nu_i/N_i=\lct_a(X,D)\) by \cite[Theorem 3.3]{V95}. Hence, by \cite[Remark 3.4]{V95} there is an exceptional curve \(E_\ell\) containing at least 3 special points such that \(\nu_i/N_i=\nu_\ell/N_\ell\). Moreover, as a special case of Proposition \ref{prop:lctpole}, it follows in this case that \(\cL^{-s_0}\) is a pole of \(Z_{f,a}^{\bir}(T)\) (of order 2).
	
	Next, for an exceptional curve \(E_{\exc}\) in \(\mu_{\dlt}^{-1}(a)\) with numerical data \((\nu,N)\), we can from now on assume that \(\nu/N\neq\nu_i/N_i\) for all \(i\in\{1,\ldots,r\}\), where \(r\) denotes the number of intersections of \(E_{\exc}\) with other components of $\mu^{-1}_{\dlt}(D)$, say \(E_1,\ldots,E_r\) (not necessarily distinct). Denote by \(t\) the number of singular points of \(Y _{\dlt}\) on \(E_{\exc}\). The number of special points on \(E_{\exc}\) is then given by \(r+t\). We consider the normalized residue of \(\mathrm{Contr}_a(E_{\exc})\) at \(\cL^{\nu/N}\) for all possible values of the tuple \((r,t)\).
	 
	 \smallskip
	 \noindent
	{\it Case \(r=0\), \((r,t)=(1,0)\), or \((r,t)=(1,1)\)}. This case cannot occur, as \(E_{\exc}\) is assumed to be an exceptional curve and its strict transform in \(\mu _{\min}^{-1}(a)\) cannot be part of a maximally admissible twig.

	 \smallskip
	 \noindent
	{\it	Case \((r,t)=(2,0)\)}. In this case, we have that \(\kappa N=N_1+N_2\) and \(\kappa\nu=\nu_1+\nu_2\) by Lemma \ref{lem:numdata}. Then
$$
\mathrm{Contr}_a(E_{\exc})=\frac{\cL^2}{\cL^\nu T^{-N}-1}\left(1+\frac{1}{\cL^{\nu_1}T^{-N_1}-1}+\frac{1}{\cL^{\nu_2}T^{-N_2}-1}\right) =$$
			$$=\frac{\cL^2\left(\cL^{\nu_1+\nu_2}T^{-N_1-N_2}-1\right)}{(\cL^\nu T^{-N}-1)(\cL^{\nu_1}T^{-N_1}-1)(\cL^{\nu_2}T^{-N_2}-1)} =\frac{\cL^2\left(1+\cL^\nu T^{-N}+\ldots+\cL^{(\kappa-1)\nu}T^{-(\kappa-1)N}\right)}{(\cL^{\nu_1}T^{-N_1}-1)(\cL^{\nu_2}T^{-N_2}-1)}.
$$
Computing \(\mathcal{R}_a(E_{\exc},\xi)\) for any \(N\)-th root of unity \(\xi\in\bC\) yields \(0\), so  $E_{\exc}$ does not contribute to any pole $\xi\cL^{{\nu}/{N}}$ of $Z^{\bir}_{f,a}(T)$.
	
	\smallskip
	At this point, the direct implication already follows from the observations above. In view of the converse implication, we define the morphism of $\bC$-algebras
	\[\varphi:\bC\left[\cL^{|\alpha_i|}, \left(1-\cL^{|\al_i|}\right)^{-1}\,\,\middle|\,\, 1\leq i\leq r\right]\to\bC,\quad \cL\mapsto0\]
	to show that the the normalized residues of \(\mathrm{Contr}_a(E_{\exc})\) at the candidate poles $\cL^{\nu/N}$ of $Z^{\bir}_{f,a}(T)$ do not add up to zero in the remaining cases.

 \smallskip
	 \noindent
	{\it	Case \(r+t\geq3\) and \(\alpha_i>0\) for all \(i\in\{1,\ldots,r\}\)}. In this case, \({\nu}/{N}=\min_{i\in S}{\nu_i}/{N_i}\).
		\begin{itemize}
			\item \(r=1\). By \cite[Proposition 3.6]{V95}, \(E_{\exc}\) is now the only component of $\Delta_{\dlt}$ contributing to the candidate pole $\cL^{\nu/N}$. Using equation (\ref{eqRes}), it follows that \(\mathcal{R}_a(E_{\exc})=\cL^{\alpha_1}/(N(\cL^{\alpha_1}-1))\neq0\).
			\item \(r>1\). Using equation (\ref{eqRes}), we have \(\varphi\left(\mathcal{R}_a(E_{\exc})\right)=({1-r})/{N}<0\).
		\end{itemize}
	
 \smallskip
	 \noindent
	{\it	Case	 \(r+t\geq3\) and \(\alpha_1<0\)}. In this case, \({\nu}/{N}\neq\min_{i\in S}{\nu_i}/{N_i}\).
		\begin{itemize}
			\item \(r=1\). By \cite[Proposition 3.6]{V95}, this contradicts the assumption that \(\alpha_1<0\), so this case does not occur.
			\item \(r=2\). Using equation (\ref{eqRes}), we have
$$				\mathcal{R}_a(E_{\exc})=\frac{1}{N}\left(1+\frac{\cL^{-\alpha_1}}{1-\cL^{-\alpha_1}}+\frac{1}{\cL^{\alpha_2}-1}\right) 
				=\frac{\cL^{\alpha_2}-\cL^{-\alpha_1}}{N(1-\cL^{-\alpha_1})(\cL^{\alpha_2}-1)}.
$$			Using Lemma \ref{lem:numdata} (iii-iv), we derive that \(\alpha_1+\alpha_2>0\), and thus that
			$\varphi\left({\mathcal{R}_a(E_{\exc})}/{\cL^{-\alpha_1}}\right)={1}/{N}>0.$	
			\item \(r>2\).   Using equation (\ref{eqRes}), we have \(\varphi\left({\mathcal{R}_a(E_{\exc})}\right)=({2-r})/{N}<0.\)
		\end{itemize}

	Lastly, when \(E_{\str}\) is an analytically irreducible component of the strict transform of \(D\), it follows from equation (\ref{eqResStrict}) that \(\mathcal{R}_a(E_{\str})={1}/({\cL^{\alpha_1}-1})\neq 0\).
	
	Assume now that \(s_0=-{\nu_i}/{N_i}\) for some exceptional curve \(E_i\) in \(\mu_{\dlt}^{-1}(a)\) for which \(r+t\geq3\), or that \(s_0=-1\). In the first case, one easily verifies that \(s_0>-1\). Hence, the last case only occurs when \(E_i\) is an analytically irreducible component of the strict transform of \(D\). It now easily follows from our computations above that in both cases the normalized residues at \(\cL^{-s_0}\) can never add up to \(0\). This finishes the proof of the converse claim.
\end{proof}
\begin{corollary}\label{cor:topbirpoles}
	We have that \(s_0\) is a pole of \(Z_{f,a}^{\bir}(T)\) if and only if \(s_0\) is a pole of \(Z_{f,a}^{\top}(s)\). In particular, the local version of the monodromy conjecture for \(Z_{f,a}^{\bir}(T)\) holds for plane curves.
\end{corollary}
\begin{proof}
This follows directly by combining Theorem \ref{thm:polestop} and Theorem \ref{thm:polescurves}. Combining this with \cite{Lo}, it follows that the local version of the birational monodromy conjecture holds for plane curves.
\end{proof}

\end{document}